\documentclass[11pt]{amsart}
\usepackage{amsthm}
\usepackage{mathtools,amssymb,latexsym,graphics,enumerate}
\usepackage[mathscr]{eucal}
\usepackage{amsmath,amsfonts,amsthm,amssymb}
\usepackage{color}
\usepackage{hyperref}
\usepackage[left=1in,right=1in,top=1in,bottom=1in]{geometry}

\numberwithin{equation}{section}

\newcommand{\rr}{\mathbb{R}}

\newcommand{\lan}{\langle}
\newcommand{\ran}{\rangle}
\newcommand{\be}{\begin{eqnarray*}}
\newcommand{\bel}{\begin{eqnarray}}
\newcommand{\ee}{\end{eqnarray*}}
\newcommand{\eel}{\end{eqnarray}}
\newcommand{\ba}{\begin{aligned}}
\newcommand{\ea}{\end{aligned}}
\newcommand{\de}{\Delta}
\newcommand{\al}{\alpha}
\newcommand{\na}{\nabla}

\newcommand{\pa}{\partial}
\newcommand{\wh}{\widehat}
\newcommand{\wt}{\widetilde}

\newcommand{\xp}[1]{^{{#1}}}

\newcommand{\myb}[1]{}

\newtheorem{theorem}{Theorem}

\newtheorem{lem}{Lemma}
\newtheorem{pro}{Proposition}
\newtheorem{rmk}{Remark}

\numberwithin{cor}{section}
\numberwithin{pro}{section}
\numberwithin{rmk}{section}
\numberwithin{remark}{section}
\numberwithin{lem}{section}

\newcommand{\norm}[1]{\left\lVert#1\right\rVert}

\newcommand{\grad}{\nabla}

\newcommand\Torus{{\mathbb T}}

\newcommand{\dss}{\displaystyle}

\mathtoolsset{showonlyrefs=true}

\title[Enhanced dissipation, hypoellipticity for passive scalar equations]{Enhanced dissipation, hypoellipticity for passive scalar equations  with fractional dissipation}
\dedicatory{Dedicated to Yijun He}
\date{\today}

\author{Siming He} \thanks{simhe@math.duke.edu, Department of Mathematics, Duke University}

\thanks{\textbf{Acknowledgment.} This work was supported in part by NSF grants DMS 2006660, DMS grant 2006372. The author would like to thank Alexander Kiselev for many discussions. { The author would also like to thank the referees for pointing out important generalizations of the main result.}}
\begin{document}
\begin{abstract}
We consider the passive scalar equations subject to shear flow advection and  fractional dissipation. The enhanced dissipation estimates are derived. For the classical passive scalar equation ($\gamma=1$), our result agrees with the sharp one obtained in \cite{Wei18}.
\end{abstract}
\maketitle
\section{Introduction}
We consider the passive scalar equations subject to shear advection and fractional dissipation:
\begin{align}\label{EQ:Fractional_PS}
\pa_t n+u(y)\pa_x n=&-\nu(-\de_x)^\gamma n-\nu (-\de_y)^\gamma n,\\
\quad n(t=0,x,y)=&n_0(x,y),\quad  (x ,y)\in\Torus^2=[-\pi,\pi]^2.
\end{align} 
Their hypoelliptic counterparts  read as follows%
\begin{align}
\pa_t \eta +u(y)\pa_x \eta=&-\nu (-\de_y)^\gamma \eta,\label{EQ:Fractional_Hypoelliptic}\\
\quad \eta(t=0,x,y)=&\eta_0(x,y),\quad  (x ,y)\in\Torus^2.
\end{align} 
Here $n,\,\eta$ denote the densities transported by the flow. The fractional dissipation order $\gamma$ takes value in $(0,2]$. The viscosity $\nu$ is   small, i.e., $\nu\in(0,1)$. {Since the dynamics \eqref{EQ:Fractional_PS}, \eqref{EQ:Fractional_Hypoelliptic} preserve the  average of solutions, one can subtract the average and assume without loss of generality that (\cite{ElgindiCotiZelatiDelgadino18})} 
\begin{align}\label{zero_average}
\int_{\Torus}n_0(x,y)dx=0,\quad \int_{\Torus}\eta_0(x,y)dx=0,\quad \forall y\in\Torus.
\end{align} 
Assume the shear flow profiles $u(y)$ have finitely many critical points $\{y_i^\star\}_{i=1}^N$.  The \emph{vanishing order} $j_i$ associated with each critical point $y_i^\star$ is defined as the smallest integer such that
\begin{align}\label{vanishing_order}
u^{(\ell)}&(y_i^\star)=0,\quad u^{(j_i+1)}(y_i^\star)\neq 0,\quad \forall 1\leq \ell \leq j_i,\,  \ell\in \mathbb{N}.
\end{align} 
The \emph{maximal vanishing order}  $j_m $ of the shear flow profile $u(y)$ is  $j_{m}:=\max_{i=1}^N \{j_i\}$.  Since any smooth shear flow profiles on the torus $\Torus$ have at least one critical point, the maximal vanishing orders $j_m $ are greater than $1$. If the maximal vanishing order is $1$, the shear flow is  \emph{nondegenerate}.

The enhanced dissipation effect of the classical passive scalar equations ($\gamma=1$) subject to shear flow has attracted much attentions from the mathematical fluid mechanics community in the recent years. In the paper \cite{BCZ15}, J. Bedrossian and M. Coti-Zelati applied hypocoercivity functional (\cite{villani2009, BeckWayne11}) to show that if $\nu$ is smaller than a universal threshold $\nu_0$,   the following enhanced dissipation estimate holds for some universal constants $C> 1 ,\, \delta_{ED}\in(0,1) $,
\begin{align}
\|\eta(t)\|_{L^2}\leq C  \|\eta_0\|_{L^2} e^{-\delta_{ED} d(\nu)|\log\nu|^{-2} t},\quad d(\nu)=\frac{j_m +1}{j_m +3},\,\forall t\in[0,\infty).\label{PS_ED_classical}
\end{align}   

Their result was later improved by D. Wei \cite{Wei18}. Combining the resolvent estimates and a Gearhart-Pr\"uss type theorem, D. Wei removed the logarithmic correction in the dissipation rate. Later, M. Coti Zelati and T. Drivas  showed that the enhanced dissipation rate $d(\nu)$ appeared in  \eqref{PS_ED_classical} is sharp (\cite{CotiZelatiDrivas19}).  
The underlying mechanism of the enhanced dissipation effect is that the shear flow advection triggers the phase mixing phenomenon (\cite{MouhotVillani11, Zillinger2014, BM13}), which amplifies the damping effect of the dissipation operators (see, e.g., \cite{BMV14, ElgindiCotiZelatiDelgadino18}). Similar phase mixing phenomena play a fundamental role in  Landau damping, see, e.g., \cite{MouhotVillani11,  BMM13, Bedrossian21}. 
Enhanced dissipation effect of the rough shear flows, and its relation to mixing are explored in \cite{Wei18, ColomboCotiZelatiWidmayer20}.

The shear flows' enhanced dissipation effect has found applications in various problems in fluid mechanics, plasma physics, and biology. First of all, the shear flows' enhanced dissipation is one of the stabilizing mechanisms in hydrodynamic stability. We refer the interested readers to the study of stability of the Couette flows (\cite{Romanov73, BMV14, BGM15I, BGM15II, BGM15III, ChenLiWeiZhang18, BedrossianHe19}), the {P}oiseuille flows (\cite{CotiZelatiElgindiWidmayer20, DingLin20}), and the Kolmogorov flows (\cite{WeiZhangZhao15, IbrahimMaekawaMasmoudi19, LiWeiZhang20}). In plasma physics, the enhanced collision effect, equivalent to the enhanced dissipation effect, stabilizes the plasma and prevents the echo-chain instability (\cite{Bedrossian17}). In biology, the enhanced dissipation effect of the ambient shear flows suppresses the chemotactic blow-ups (\cite{BedrossianHe16, He}).

The enhanced dissipation effect of the shear flow is \emph{heterogeneous}. If the initial data $n_0$ of the passive scalar equation depends only on $y$-variables, $n(t,y)$ solves the heat equation, and no enhanced dissipation is {possible}.  Hence the zero-average constraint \eqref{zero_average} is enforced. However, there exist fluid flows inducing the enhanced dissipation effect in all directions. These are the relaxation-enhancing flows. The concept is first introduced by P. Constantin et al.,  \cite{ConstantinEtAl08}. In the papers   \cite{ElgindiCotiZelatiDelgadino18, FengIyer19}, the authors prove that flows with mixing properties are relaxation enhancing.  Explicit constructions of mixing flows have attracted much attention in the dynamical system and fluid mechanics community, see, e.g., \cite{vonNeumann32, Kolmogorov53, Seis13, IyerKiselevXu14, AlbertiCrippaMazzucato14, AlbertiCrippaMazzucato19, YaoZlatos17,  ElgindiZlatos18}, and the references therein. The relaxation enhancing flows find applications in various problems, see, e.g., \cite{KiselevXu15, HopfRodrigo18, IyerXuZlatos, FengFengIyerThiffeault20}.

Much less is known for the systems \eqref{EQ:Fractional_PS} and  \eqref{EQ:Fractional_Hypoelliptic}. The enhanced dissipation result for the passive scalar equation \eqref{EQ:Fractional_PS} is obtained in \cite{ElgindiCotiZelatiDelgadino18}. However, the enhanced dissipation rate obtained is not sharp in general. Recently, an enhanced dissipation estimate for the $\gamma=2$ case is derived in \cite{CotiZelatiDolceFengMazzucato}.

By taking the Fourier transform in the $x$-variable, one obtains the equations for each Fourier mode:
\begin{align}
\pa_t \wh n_k+iu(y) k\wh n_k=&-\nu|k|^{2\gamma}\wh n_k-\nu (-\de_y)^\gamma \wh n_k,\quad \wh n_{k}(t=0,y)=\wh n_{0;k}(y);\label{EQ:Fractional_k_by_k}\\
\pa_t \wh \eta_k +iu(y) k\wh \eta_k=&-\nu (-\de_y)^\gamma \wh \eta_k,\quad \wh \eta_{k}(t=0,y)=\wh \eta _{0;k}(y).\label{EQ:Fractional_Hypoelliptic_k_by_k}
\end{align} 

The first main theorem of the paper is the following
\begin{theorem}\label{thm:semigroup_est_PS}
Consider the equation \eqref{EQ:Fractional_Hypoelliptic} subject to initial condition $\eta_{0 }\in L^2(\Torus^2)$. Assume that the shear flow profile $u(\cdot)\in C^\infty(\Torus)$ has finitely many critical points $\{y_i^\star\}_{i=1}^N$ and the  maximal vanishing order $j_m\geq1$ is finite. Further assume that there exist $R_i\in(0,\frac{\pi}{10}),\, i\in\{1,2,...N\}$ such that in the neighborhood $B(y_i^\star;R_i)\subset \Torus$, the following estimate holds for some universal constant $C_1(u)>1$, 
\begin{align}\label{relation:u'_and_y-y_i_star_j}
\frac{1}{C_1(u)}|y-y_i^\star|^{j_i}\leq |u'(y)|\leq C_1(u) |y-y_i^\star|^{j_i},\quad \forall y \in B(y_i^\star;R_i). 
\end{align} 
Then there exists a viscosity threshold $\nu_0=\nu_0(u)$ such that if $\nu\leq \nu_0$, the following enhanced dissipation estimate holds,  
\begin{align}\label{ED_full}
||\eta (t)\|_2\leq C  \|\eta_{0 }\|_2e^{-\delta_{ED} d (\nu)|\log\nu|^{-\beta(\gamma)}  t},\quad \forall t\geq 0,\quad d(\nu)=\frac{j_m +1}{j_m +1+2\gamma},\quad \gamma\in (1/2,2].
\end{align}
Here the constants $C> 1 ,\,\delta_{ED}>0$ depend only on the shear profile $u$. The parameter $\beta=\beta(\gamma)$ depends on the fractional dissipation order $\gamma$ and vanishes for ${\gamma\in[1,2]}$. The explicit form of $\beta$ is $\beta(\gamma)=8\gamma(1-\gamma)\mathbf{1}_{\gamma\in(1/2,1)}$.  

For the $k$-by-$k$ system \eqref{EQ:Fractional_Hypoelliptic_k_by_k}, the following enhanced dissipation estimate holds
\begin{align}\label{ED_k_by_k}
\|\wh\eta_k(t)\|_2\leq C \|\wh\eta_{0;k}\|_2 e\xp{-\delta_{ED}\nu\xp{\frac{j_m +1}{j_m +1+2\gamma}}|k|\xp{\frac{2\gamma}{j_m +1+2\gamma}}|\log(\nu|k|^{-1})|^{-\beta(\gamma)} t},\,\forall t\geq  0,\, k\neq0, \, \gamma\in(1/2,2].
\end{align}
\end{theorem}
\begin{rmk}
The $k$-by-$k$ estimate \eqref{ED_k_by_k} implies that the solution to the hypoelliptic equation \eqref{EQ:Fractional_Hypoelliptic} gain Gevrey regularity in $x$-direction instantly. This gain in regularity is related to Hormander's hypoellipticity theorem, see, e.g.,  \cite{Kolmogoroff34,Hormander67, Nualart06, Malliavin97}.
\end{rmk}
\begin{rmk}
If the shear flow $u$ is analytic near the critical points $y_i^\star$, then the condition \eqref{relation:u'_and_y-y_i_star_j} holds.
\end{rmk}\begin{rmk}{Our argument does not provide the enhanced dissipation estimate in the regime $\gamma\in(0,1/2]$. The main reason is that our proof requires an apriori $L^\infty$-bound of the solutions to the resolvent equation. If  $\gamma\in(0,1/2]$, Sobolev embedding does not guarantee such $L^\infty$-estimate. However, a recent manuscript \cite{LiZhao21} seems to suggest that the enhanced dissipation estimate with rate $\nu^{\frac{j_m+1}{j_m+1+2\gamma}}$ might still hold in the range $\gamma\in(0,1/2]$. We will leave that as a conjecture to pursuit in the future.
}
\end{rmk}
\begin{rmk}{
If the shear flow profile is non-degenerate, i.e., $j_m=1$, then the enhanced dissipation rate (modulo logarithmic correction) is $\nu^{\frac{1}{1+\gamma}}$. When $\gamma=1$, we recover the classical rate $\nu^{1/2}$.}
\end{rmk}\begin{rmk}
The logarithmic loss here for $\gamma\in(1/2,1)$ comes from the estimation of the $\dot H^{1/2}$-semi-norm of specific functions. New ideas are needed to drop the logarithmic factor or extend the result to $\gamma\in(0,1/2]$. If $\gamma$ ranges from $[1,2]$, the $\dot H^1$-norm will be applied instead and no loss of $|\log (\nu|k|^{-1})|^\beta$ will appear in the dissipation rate \eqref{ED_k_by_k}. 
\end{rmk}

The second main theorem provides the enhanced dissipation for the equations \eqref{EQ:Fractional_PS} and \eqref{EQ:Fractional_k_by_k}.
\begin{theorem}\label{thm:semigroup_est_PS_full}
Consider the equation \eqref{EQ:Fractional_PS} subject to initial condition $n_{0}\in L^2(\Torus^2)$. Assume the conditions in Theorem \ref{thm:semigroup_est_PS}. Then there exists a viscosity threshold $\nu_0=\nu_0(u)$ such that if $\nu\leq \nu_0$,  the following enhanced dissipation estimate holds
\begin{align}\label{full_ED}&&
\|n (t)\|_2\leq C\|n_{0 }\|_2  e^{-\nu t-\delta_{ED} d (\nu)|\log\nu|^{-\beta(\gamma)} t},\quad \forall t\geq 0,\quad d(\nu)=\frac{j_m +1}{j_m +1+2\gamma},\, \gamma\in (1/2,2].
\end{align}
Here  the  constants $C>1 ,\,\delta_{ED}>0$ depend only  on the shear profile $u$. The parameter $\beta(\gamma)=8\gamma(1-\gamma)\mathbf{1}_{\gamma\in(1/2,1)}$ vanishes for $\gamma\in[1,2]$. 

For the $k$-by-$k$ system \eqref{EQ:Fractional_k_by_k} ($k\neq0$), the following estimate holds for constant $C,\, \delta_{ED}$ which only depend on $u(\cdot)$,
\begin{align}\label{full_ED_k_by_k}
&&\|\wh n_k(t)\|_2\leq C \|\wh n _{0;k}\|_2 e\xp{-\nu|k|^{2\gamma}t-\delta_{ED}\nu\xp{\frac{j_m +1}{j_m +1+2\gamma}}|k|\xp{\frac{2\gamma}{j_m +1+2\gamma}}|\log(\nu|k|^{-1})|^{-\beta(\gamma)} t},\,\forall t\geq 0,\,\gamma \in(1/2,2].
\end{align} 
\end{theorem}

{\begin{rmk}
It is worth noting that our method can be adapted to provide the same enhanced dissipation estimate  for passive scalar solutions subject to the classical fraction dissipation operator $-(-\de)^\gamma=-(|\pa_x|^2+|\pa_y|^2)^{\gamma}$. Details of the adjustments are highlighted in Remark \ref{rmk:classical}.  
\end{rmk}
\begin{rmk}
Similar argument yields enhanced dissipation for shear flows whose profile $u$ are Lipschitz. Consider profile $u$ with finitely many critical points. Furthermore, assume that the absolute value of the derivatives of the profile are strictly positive whenever they exist, i.e., \\$\min_{y\in \Torus}\{|u'(y)|\,|\,u'(y) \text{ exists}.\}>c>0$. Then the enhanced dissipation estimate holds with rate $O(\nu^{\frac{1}{2\gamma+1}})$ (modulo logarithmic factors). The argument is similar to the proof of Theorem 4 in \cite{GongHeKiselev21}. %
\end{rmk}
}


Our analysis combines a spectral gap estimate in the spirit of the work \cite{BCZ15}  and the Gearhart-Pr\"uss type theorem proven in \cite{Wei18}. Furthermore, detailed resolvent estimates are carried out to prove the result. The resolvent estimate has found applications in various works in the hydrodynamics stability, see, e.g., \cite{ChenLiWeiZhang18, LiWeiZhang20, DingLin20}.

The paper is organized as follows: in section \ref{Section:2}, we present the proof of the main theorems; in section \ref{Sec:3}, we prove the main resolvent estimates (Proposition \ref{pro:resolvent_estimate}).  

\textbf{Notation: }Throughout the paper, the constant $B, \,C$ are constants independent of $\nu,\, k$ and are changing from line to line. In the section \ref{Sec:3}, the constant $C$ can depend on a small constant $\delta$ and we will specify when it happens. The constants $C_0,\,C_{1}, \, C_{\mathrm{spec}}, ...$ will be explicitly defined. The notations $T_1,\, T_2,\, T_3,\, ...$ denote terms in long expressions and will be recycled after the proof of each lemma. Hence the meanings of $T_{(\cdot)}$ change from lemma to lemma. {We use $|A|$ to denote the area of the set $A$.}

We consider the Fourier transform \emph{only in the $x$ variable}, and denote it and its inverse as
\begin{equation*}
\wh{f}_k(y) := \frac{1}{2\pi}\int_{-\pi}^\pi e^{-ikx} f(x,y) dx , \quad \check{g}(x,y) := \sum_{k=-\infty}^\infty g_k(y) e^{ikx}.
\end{equation*} 
If the function only depends on the $y$-variable, we use similar formulas to calculate the Fourier transform/inverse transform in $y$. 
The symbol $\lan{f}\ran$ represents average on the torus $\Torus$, i.e., $\dss\lan{f}\ran=\frac{1}{2\pi}\int_{\Torus} f(y)dy$. The symbol $\overline{f}$ denotes the complex conjugate. For any measurable function $m(\xi)$, we define the Fourier multiplier $m(\grad)f := (m(\xi)\widehat{f}(\xi))^{\vee}$. 
The $L^p_y$-norms are defined as 
\begin{align}\|g\|_{L^p_y}=\left(\int_{\Torus}|g(y)|^pdy\right)^{\frac{1}{p}}, \quad p\in[1,\infty),
\end{align} with natural extension to $p=\infty$. The $\dot H^\gamma$-seminorm and the $H^\gamma$-norm are  defined as follows:
\begin{align}
\|g\|_{\dot H^\gamma_y}^2=&\||\pa_y|^\gamma g\|_{L^2_y}^2=\sum_{\ell\in \mathbb{Z}}|\ell|^{2\gamma}|\wh{g}_\ell|^2,\quad 
\|g\|_{H_y^\gamma}^2=\|g\|_{\dot H_y^\gamma}^2+\|g\|_{L_y^2}^2.&
\end{align}
\section{Proof of Theorem \ref{thm:semigroup_est_PS} and Theorem \ref{thm:semigroup_est_PS_full}}\label{Section:2}
In this section, we prove Theorem \ref{thm:semigroup_est_PS} and Theorem \ref{thm:semigroup_est_PS_full}. The main goal is to derive the $k$-by-$k$ estimate \eqref{ED_k_by_k}, where $k$ is horizontal wave number. 

We make two preparations for the proof of the main theorem \ref{thm:semigroup_est_PS}. First, we reduce the problem for general wave number $k\in \mathbb{Z}\backslash \{0\}$ to the case $k=1$. Secondly, we present a semigroup estimate from \cite{Wei18}. 

If the estimate \eqref{ED_k_by_k} is proven for $k=1$, then by changing the sign of the shear $u(y)$, we obtain the estimate for $k=-1$. Now we consider the general case $k\in \mathbb{Z}\backslash\{ 0\}$ and rewrite the equation \eqref{EQ:Fractional_Hypoelliptic_k_by_k} as follows
\begin{align}
\frac{1}{|k|}\pa_t \wh \eta_k+iu(y)\frac{k}{|k|}\wh \eta_k=-\frac{\nu}{|k|}(-\de_y)^{\gamma}\wh \eta_k.%
\end{align}
By rescaling in time $\tau=|k|t$ and setting $\wt\nu = \frac{\nu}{|k|}$, we obtain that
\begin{align*}
\pa_\tau \wh \eta_k +iu(y)\frac{k}{|k|}\wh \eta_k=-\wt \nu(-\de_y)^{\gamma}\widehat \eta_k.
\end{align*}
Application of the enhanced dissipation estimate \eqref{ED_k_by_k} for $k=\pm 1$ yields the following estimate 
\begin{align}
\|\wh \eta_k(t)\|_2=\|\wh \eta_k(\tau)\|_2\leq C \|\wh \eta_{0;k}\|_2e^{-\delta_{ED} d(\wt \nu)|\log\wt \nu|^{-\beta(\gamma)}\tau }=C \|\wh \eta_k(0)\|_2e^{-\delta d( \nu|k|^{-1}) |k| |\log (\nu|k|^{-1})|^{-\beta(\gamma)}t}.
\end{align}
Now recalling the definition of $ d(\cdot )$ \eqref{ED_full} yields  the estimate  \eqref{ED_k_by_k} for  general $k\neq 0$. Hence in the remaining part of the paper, we focus on the $k=1$ case and drop the subscript $k$.

The second preparation involved in the proof is the Gearhart-Pr\"uss type theorem proven by D. Wei, \cite{Wei18}. The theorem translates spectral estimate into quantitative semigroup estimate under suitable conditions.  We  recall the key concepts in the paper \cite{Wei18}. Let $X$ be a complex Hilbert space. Let $H$ be a linear operator in $X$ with domain $D(H)$. Denote $\mathcal{B}(X)$ as the space of bounded linear operators on $X$ equipped with operator norm $\|\cdot\|$ and $I$ as the identity operator.  A closed operator $H$ 
is $m$-accretive if the set $\{\lambda|\mathrm{Re}\lambda<0\}$ is contained in the resolvent set of $H$, and 
\begin{align}
(H+\lambda I)^{-1}\in \mathcal {B}(X), \quad\| (H+\lambda I)^{-1}\|\leq (\mathrm{Re}\lambda)^{-1},\quad \forall \,\mathrm{Re}\lambda>0.
\end{align}
An $m$-accretive operator $H$ is accretive and densely defined. 
The $-H$ is a generator of a semigroup $e^{-tH}$. The decay rate of the semigroup $e^{-tH}$ is determined by the   following quantity
\begin{align}\label{Psi_H}
\Psi(H):=\inf \{\|(H-i\lambda I)f\|_X; f\in D(H), \lambda\in \rr, \|f\|_X=1\}.
\end{align} This is the content of the main theorem of the paper \cite{Wei18}.   
\begin{theorem}\label{thm:m_accretive}
Assume that $H$ is an $m$-accretive operator in a Hilbert space $X$. Then the following estimate holds:
\begin{align}
\|e^{-tH}\| \leq e^{-t\Psi(H)+\pi/2}, \quad \forall t\geq 0.
\end{align}
\end{theorem}

We define the function space $X$ to be $X=L^2$ and the differential operator to be 
\begin{align}{
H := \nu(-\de_y)^{\gamma} +iu(y),\quad \gamma\in(1/2,2].}
\end{align} 
The domain of the operator is $D(H)=H^{2\gamma}(\mathbb{T})$.  By testing the equation $(H+\lambda I) w=f$ by $w$ and taking the real part, we obtain that
\begin{align}
\nu \| |\pa_y |^\gamma w\|_2^2+\mathrm{Re}\lambda \|w\|_2^2=\mathrm{Re}\int f\overline{w}dy.
\end{align}
Hence if the real part of the spectral parameter $\mathrm{Re}\lambda>0$, we have that $\mathrm{Re}\lambda\|w\|_2\leq \|f\|_2$, which in turn yields that $\|(H+\lambda I)^{-1}\|\leq (\mathrm{Re}\lambda)^{-1}, \, \mathrm{Re}\lambda>0$. Therefore, the operator $H$ is $m$-accretive. 
 
Now we are ready to prove the key estimate \eqref{ED_k_by_k}.  
To apply Theorem \ref{thm:m_accretive}, we consider the following resolvent equation associated with the hypoelliptic passive scalar equation \eqref{EQ:Fractional_Hypoelliptic_k_by_k}
\begin{align}\label{Resolvent}
(H-i\lambda ) w= \nu (-\de_{y})^{\gamma}w+i (u(y)-\lambda) w=F .
\end{align}
Recall that the shear flow profile $u(\cdot)$  has critical values $\{u_{i}^\star=u(y_i^\star)\}_{i=1}^N$, which locate at critical points $\{y_i^\star\}_{i=1}^N$ with vanishing order $\{j_i\}_{i=1}^N$  \eqref{vanishing_order}.  We present the  following proposition, whose proof is postponed to Section \ref{Sec:3}.
\begin{pro}\label{pro:resolvent_estimate}
Consider the resolvent equation \eqref{Resolvent}. Assume conditions in Theorem \ref{thm:semigroup_est_PS}. Further assume that the spectral parameter $\lambda$ in \eqref{Resolvent} ranges on the real line $\rr$. Then  the following resolvent estimate holds if $\nu>0$ is smaller than a threshold $\nu_0=\nu_0(u)$,
\begin{align}
\nu^{\frac{1}{2}(1+\frac{j_m +1}{j_m +1+2\gamma})} \||\pa_y |^\gamma w\|_2 +\nu^{\frac{j_{m}+1}{j_{m}+1+2\gamma }} \|w\|_2\leq& C_\star(u) \|F\|_2,\quad \gamma\in[1,2];\\
\nu^{\frac{1}{2}(1+\frac{j_m +1}{j_m +1+2\gamma})} \||\pa_y |^\gamma w\|_2 +\nu^{\frac{j_{m}+1}{j_{m}+1+2\gamma }} \|w\|_2\leq& C_\star(u) |\log\nu|^{8\gamma(1-\gamma)}\|F\|_2,\quad \gamma\in(1/2,1).\label{Goal} 
\end{align}
Here $j_m $ is the maximal vanishing order of all the critical points \eqref{vanishing_order}, i.e., $j_m =\max_{i\in\{1,...,N\}}j_i$ . The constant $C_\star=C_\star(u)\geq 1$ depends only on the shear profile $u(\cdot)$.
\end{pro}

Combining Theorem \ref{thm:m_accretive} and Proposition  \ref{pro:resolvent_estimate} yields the $k$-by-$k$ estimate \eqref{ED_k_by_k}.  Summing up all $k$-modes yields the estimate \eqref{ED_full}. This concludes the proof of Theorem \ref{thm:semigroup_est_PS}. The proof of the estimates in Theorem \ref{thm:semigroup_est_PS_full} follows from the observation that $\wh n_k(t, y) e^{\nu|k|^{2\gamma}t}$ solves the equation \eqref{EQ:Fractional_Hypoelliptic_k_by_k}. Therefore the $L^2$-norm is bounded as in  \eqref{ED_k_by_k}, i.e.,
\begin{align}{
\|\wh n _k(t)e^{\nu|k|^{2\gamma} t}\|_2\leq C \|\wh n_{0;k}\|_2 e\xp{-\delta_{ED}\nu\xp{\frac{j_m +1}{j_m +1+2\gamma}}|k|\xp{\frac{2\gamma}{j_m +1+2\gamma}}|\log(\nu|k|^{-1})|^{-\beta} t},\quad\forall t\geq  0,\quad \gamma\in(1/2,2].}
\end{align} By multiplying both side by $e^{-\nu|k|^{2\gamma}t}$ , we obtain the  estimate \eqref{full_ED_k_by_k}. Summing up all $k$-modes yields \eqref{full_ED}. Hence the proof of Theorem \ref{thm:semigroup_est_PS_full} is completed.
{
\begin{rmk}\label{rmk:classical}
The above arguments can be adapted to treat the passive scalar solutions subject to classical fractional dissipation operator $(-\de)^{\gamma}=(|\pa_x|^2+|\pa_y|^2)^{\gamma},\, \gamma\in(1/2,2]$. One of the adjustments is that one will not re-scale the time variable to get rid of $|k|$. As a result, we consider the operator $H_{k,\nu}=\nu(|k|^2+|\pa_y|^2)^\gamma+iu(y)k$ and its resolvent. The constructions of augmented functions in the next section are similar. We refer the interested readers to the appendix of \cite{CotiZelatiDolceFengMazzucato} for the treatment in the bi-Laplacian case. 
\end{rmk}}
\section{Proof of Proposition \ref{pro:resolvent_estimate}} 
\label{Sec:3}
In this section, we prove Proposition \ref{pro:resolvent_estimate}. 

Following the paper \cite{BCZ15}, we first introduce a partition of unity on the torus and localize the solution $w$ to \eqref{Resolvent} around each critical point $y_i^\star$. To this end, we consider $2r_i$-neighborhood $B(y_i^\star;2r_i)$ around each critical point $y_i^\star$ for $0<r_i\leq \frac{1}{10}\pi$. Further assume that the dilated balls $\{B(y_i^\star;4r_i) \}_{i=1}^N$ are pair-wise disjoint. Next we define $\{\xi_i^2\}_{i=0}^N$ to be a partition of unity on the torus such that $\xi_i\in C^\infty(\mathbb{T})$ for $i\in \{0,1, ...,N\}$ and $\mathrm{support}(\xi_i)=B(y_i^\star;2r_i)$ for $i\in \{1,2,...,N\} $\myb{(Check!? If we consider the standard smooth partition of unity, it has the form $\exp\{-\frac{1}{1-x^2}\}$ near the zero, and taking root will not change the qualitative behavior.)}. Moreover, for the index $i$ ranges from $1$ to $N$, $\xi_i(y)\equiv 1$ in the neighborhood $B(y_i^\star; r_i)$ and decays to zero as $y$ approaches $  \pa B(y_i^\star;2r_i)$.  
The function $\xi_0^2=1-\sum_{i=1}^N\xi_i^2$ has support away from all the critical points and hence $\min_{y\in \mathrm{support}\xi_0}|u'(y)| \geq \frac{1}{C(u)}$ for some positive constant $C(u)>0$ which is independent of $\nu$. 
At each critical point $(y_i^\star, u_i^\star)$ with vanishing order $j_i$, by Taylor's theorem, the shear flow profile has the following expansion
\begin{align}
u(y)-u_i^\star =\frac{u^{(j+1)}(y_i^\star)}{(j+1)!}(y-y_i^\star)^{j+1}+O(|y-y_i^\star|^{j+2}).
\end{align}
We choose the radius $r_i$ small enough such that on the neighborhood $B(y_i^\star;2r_i)$, the first term on the right hand side dominates, i.e.,
\begin{align}
\frac{1}{C_0(u)}|y-y_i^\star|^{j+1}\leq |u(y)-u_i^\star|\leq C_0(u) |y-y_i^\star|^{j+1},\quad \forall y\in\mathrm{support} (\xi_i),\, i\in\{1,2,...,N\}.\label{non-degenerate_0}
\end{align}
Here the constant $C_0\geq 1$ only depends on the shear profile.  Moreover, we can choose $r_i\leq \frac{1}{4}R_i$ in \eqref{relation:u'_and_y-y_i_star_j}, such that on the support of $\xi_i,\, i\neq 0$, the following relation holds
\begin{align}
\frac{1}{C_1(u)}|y-y_i^\star|^j\leq |u'(y)|\leq C_1(u) |y-y_i^\star|^j,\quad \forall y \in\mathrm{support} (\xi_i),\, i\in\{1,2,...,N\} . \label{non-degenerate}
\end{align}
Since the above choice of $r_i$ depends only on the shear profile $u$, there exists a constant $C(u)$, which is independent of the viscosity $\nu$, such that the following estimate holds
\begin{align}\label{Est_xi_i}
\|\xi_i\|_{W^{4,\infty}(\Torus)}\leq C(u),\quad \forall  i\in \{0,1,2,..., N\}.
\end{align}

Next we present some energy relations associated to the equation \eqref{Resolvent} and specify the interesting range of the spectral parameter $\lambda$. By testing the equation \eqref{Resolvent} against the conjugate $\overline{w}$ and taking the real and imaginary part, we obtain that
\begin{align}
\nu  \||\pa_y|^\gamma w\|_2^2  =\mathrm{Re}\int_\Torus F \overline{w}dy\leq& \|F\|_2\|w\|_2\label{Real}\\
\int_\Torus (u(y)-\lambda)|w|^2dy=&\mathrm{Im}\int_\Torus F \overline{w}dy,\label{Imaginary}
\end{align} 
Testing the equation \eqref{Resolvent} with $g \overline{w}$, where $g$ is any smooth real-valued function on $\Torus$,  yields the following equation  
\begin{align}
\int_\Torus (u(y)-\lambda)g|w|^2dy = &\mathrm{Im}\int _\Torus F \overline{w}g dy-\nu\mathrm{Im} \int_\Torus |\pa_y|^\gamma w |\pa_y|^\gamma (g\overline{w})dy.\label{Imaginary_with_f}
\end{align}
Direct application of these energy equalities ensures the estimate \eqref{Goal} given that the spectral parameter $\lambda$ is away from the range of $u(\cdot )$. This is the content of the next lemma.
\begin{lem}\label{lem:away_from_range}
Assume that the spectral parameter $\lambda\in\rr$ is away from the range of shear profile $u(\cdot )$ in the sense that 
\begin{align}
\min_{y\in \Torus}|\lambda-u(y)|\geq \delta\nu^{\frac{j_m +1}{j_m +1+2\gamma}} .\label{away_from_range} 
\end{align}
Here the small parameter $\delta\in(0,1)$ is independent of the viscosity $\nu$ and $j_m$ is the maximal vanishing order of the shear $u$. Then following estimate holds
\begin{align}
\nu^{1+\frac{j_m +1}{j_m +1+2\gamma}} \||\pa_y |^\gamma w\|_{L^2_y}^{2}+\nu^{2\frac{j_m+1}{j_m+1+2\gamma}}\|w\|_{L^2_y}^2\leq C(\delta^{-1})\|F\|_{L^2_y}^2.
\end{align}
\end{lem}
\begin{rmk}
The parameter $\delta=\delta(u)>0$ will be chosen in \eqref{Choice_of_delta}. Hence the estimate we obtain is consistent with \eqref{Goal}.
\end{rmk}
\begin{proof} 
The $L^2$-estimate is the key.  
Applying the relation  \eqref{Imaginary}, the fact that $u(y)-\lambda$ has fixed sign under the constraint \eqref{away_from_range}, and the  H\"older inequality yields the following estimate 
\begin{align}
\delta\nu^{\frac{j_m +1}{j_m +1+2\gamma}} \|w\|_2\leq C \|F\|_2.\label{L2_est_away_from_range}
\end{align}
Combining the relation \eqref{Real} and the $L^2$-estimate \eqref{L2_est_away_from_range}, we have the higher regularity norm estimate
\begin{align}\label{H1_est_away_from_range}
\nu\||\pa_y|^\gamma w\|_2^2 \leq C\delta^{-1} \nu^{-\frac{j_m +1}{j_m +1+2\gamma}} \|F\|_2^2.
\end{align}
Combining the inequalities \eqref{L2_est_away_from_range} and \eqref{H1_est_away_from_range} yields the result. 
\end{proof}
Hence we focus on the case where the spectral parameter $\lambda$ is close to the range of the shear profile $u$, i.e., \begin{align}\min_{y\in\Torus}|\lambda- u(y)|\leq \delta\nu^{\frac{j_m +1}{j_m +1+2\gamma}}  .
\end{align} 
Here $\delta (u)$ is chosen in \eqref{Choice_of_delta}. 
We decompose the $L^2$-norm $\|w\|_2^2$ into $N+1$ pieces with the partition of unity $\xi_i^2$,
\begin{align}
\|w\|_2^2=\sum_{i=0}^{N}\|w\xi_i\|_2^2.
\end{align}
Our primary goal is to derive $L^2$-estimate on each component $w\xi_i,\,\forall i\in\{0,1,...,N\} $, 
\begin{align}
\|w\xi_i\|_2^2\leq &C(B,u,\delta^{-1},j_i)\nu^{-2\frac{j_{i}+1}{j_{i}+1+2\gamma}}|\log \nu|^{16\gamma\al(\gamma) }  \|F\|_2^2\\
&+\left(\frac{1}{B}+C(u,\delta^{-1}, j_i)\nu^{ \frac{1}{j_{{m}}+1+2\gamma}}|\mathrm{log}\nu|^{4\al(\gamma) }\right)\|w\|_2^2, \label{Goal_i_by_i}
\end{align}
Here the universal constant $B\geq 1$ is arbitrary and will to be determined at the end of the proof \eqref{Choice_of_B_delta_nu_0}. The $j_i,\, i\neq 0$ is the vanishing order of the critical points \eqref{vanishing_order}. If $i=0$, $j_0=0$. The $j_m$ is the maximal vanishing order. 
The parameter $\al=\al(\gamma)$ is 
\begin{align}\label{al }
\al(\gamma)=\mathbf{1}_{\gamma\in(1/2,1)}(1-\gamma). \quad
\end{align}
In the latter part of the proof, if the constant $C$ depends on $B$ or $\delta$, we will explicitly spell out.
  
To prove the estimate \eqref{Goal_i_by_i}, we first consider the components $\|w\xi_i\|_2^2$ with $i\neq 0$. We distinguish between three cases based on the relative position of $\lambda$ and the value $u_i^\star=u(y_i^\star)$ at each critical point $y_i^\star,\, \forall i\in\{1,2,...N\}$: 
\begin{align}
a) &\quad i\in \mathcal{I}_{\mathrm{near}}:\quad |\lambda-u_{i}^\star|\leq \delta \nu^{\frac{j_i+1}{j_i+1+2\gamma}} ;\label{Case a}\\
b) &\quad  i\in \mathcal{I}_{\mathrm{interm}}:\quad |\lambda-u_{i}^\star|\geq \delta\nu^{\frac{j_i+1}{j_i+1+2\gamma}} ,\quad\mathrm{dist}(y_i^\star,\{z|u(z)=\lambda\})\leq 3r_i;\label{Case b}\\ 
c) &\quad i\in \mathcal{I}_{\mathrm{far}}:\quad\mathrm{dist}(y_i^\star,\{z|u(z) =\lambda\}) \geq  3r_i .
\label{Case c}
\end{align}
Here $2r_i$ is the radius of the support of $\xi_i,\, \forall i\in\{1,2,...,N\}$. We will prove the primary estimate \eqref{Goal_i_by_i} in case $a)$, $b)$, and $c)$ in Lemma \ref{lem:case a}, Lemma \ref{lem:case b} and Lemma \ref{lem:case c}, respectively.  
Finally, we estimate the component $\|w\xi_0\|_2^2$ in Lemma \ref{lem: 0}. Once the estimate \eqref{Goal_i_by_i} is established, by summing all the contributions from different components, and taking $B^{-1}$ and then $\nu_0$ small enough, we will obtain  the estimate \eqref{Goal}.

Before proving the estimate \eqref{Goal_i_by_i} for $i\in \mathcal{I}_{\mathrm{near}}$, we introduce a crucial spectral gap estimate, which also plays a central role in \cite{BCZ15}.
\begin{lem}\label{lem:Spectral_gap}
Assume condition \eqref{relation:u'_and_y-y_i_star_j} and let $f\in H^{\gamma}(\Torus)$. Consider a critical point $y_i^\star$ of the shear flow profile $u$ with vanishing order $j=j_i\geq 1$ \eqref{vanishing_order}. The function $f_i=f\xi_i$ is supported in the $2r_i$-neighborhood of the critical point $y_i^\star$. Then the following estimate holds for some constant $C_{\mathrm{spec}}(u)\geq 1$, 
\begin{align}\label{Spectral_gap}
\nu^{\frac{j+1}{j+1+2\gamma}}\|f_i\|_{L^2(\Torus)}^2\leq C_{\mathrm{spec}}(u)\nu\|| \pa_y| ^{\gamma}f_i\|_{L^2(\Torus)}^2+C_{\mathrm{spec}}(u) \nu^{\frac{1-j}{j+1+2\gamma}}\|u' f_i\|_{L^2(\Torus)}^2,\quad \gamma\in\left(\frac{1}{2},2\right]. 
\end{align}
\end{lem}\begin{rmk}
As being discussed in the paper \cite{BCZ15} (pages $12$-$13$), the estimate \eqref{Spectral_gap} is related to the spectral gap of the differential operator $L:=(-\de_z)^\gamma+|z|^{2j}$ on $\rr$.
\end{rmk}
\begin{proof}  
First, we show that the following estimate on the torus implies \eqref{Spectral_gap},
\begin{align}\label{Spectral_gap_sigma}
\sigma^{\frac{j}{j+\gamma}}\|f_i\|_{L^2(\Torus)}^2\leq C\sigma \||\pa_y|^\gamma f_i\|_{L^2(\Torus)}^2+C\| |y-y_i^\star|^j f_i\|_{L^2(\Torus)}^2.
\end{align} 
Here the parameter $\sigma\in (0,1)$ is any small enough number. Since $f_i=f\xi_i$ is localized near the $i$-th critical point $y_i^\star$, the last term makes sense. Combining the estimate \eqref{Spectral_gap_sigma} and the condition \eqref{relation:u'_and_y-y_i_star_j}, we obtain
\begin{align}
\sigma^{\frac{j}{j+\gamma}}\|f_i\|_{L^2(\mathbb{T})}^2\leq C \sigma\||\pa_y|^\gamma f_i\|_{L^2(\Torus)}^2+C\|u 'f_i\|_{L^2(\Torus)}^2.
\end{align}
By setting $\sigma=\nu^{2\frac{j+\gamma}{j+1+2\gamma}}$ in the above  inequality, we have \eqref{Spectral_gap}.

To prove the estimate \eqref{Spectral_gap_sigma}, we first consider the $z\in \rr$ variable. On $\rr$, the following estimate holds
\begin{align}
\|g\|_{L^2(\rr)}^2\leq \int_\rr\psi_{[-2\varepsilon,2\varepsilon]}(z)|g(z)|^2dz+C(\varepsilon^{-1})\int_\rr |g(z)|^2|z|^{2j}dz,
\end{align}where $\varepsilon>0$ is a universal small constant. 
Here, $\psi_{[-2\varepsilon,2\varepsilon]}$ is a smooth cut-off function which is $1$ on $[-\varepsilon,\varepsilon]$ and has support in $[-2\varepsilon, 2\varepsilon]$. Now we make the change of variables $g(z)=f_i(\kappa^{-1}z+y_i^\star)=(f\xi_i)(\kappa^{-1}z+y_i^\star), \,\, y-y_i^\star=\kappa^{-1}z,$ for $\kappa>1$. In the $y$-coordinate, the estimate above can be rewritten as follows :
\begin{align}
\kappa\int_\rr |f_i(y)|^2dy\leq &\kappa\|f_i\|_{L^\infty(\Torus)}^2\int_\rr \psi_{[-2\varepsilon,2\varepsilon]}(\kappa (y-y_i^\star))  dy+C(\varepsilon^{-1})\kappa^{2j+1} \int_\rr |f_i(y)|^2|y-y_i^\star|^{2j}dy\\
\leq& C\varepsilon\|f_i\|_{L^\infty(\Torus)}^2+C(\varepsilon^{-1})\kappa^{2j+1}\int_\rr |f_i(y)|^2|y-y_i^\star|^{2j}dy.
\end{align} Now since $f_i$ is localized, we have that the estimate also holds with the integral domain $\rr$ replaced by  the torus $\mathbb{T}$. Since $\gamma>1/2$, the $L^\infty$-norm can be controlled through $H^\gamma$-norm. Now applying the H\"older inequality, Young inequality and Gagliardo-Nirenberg interpolation inequality yields that
\begin{align*}
\kappa&\int_{\mathbb{T}} |f_i|^2dy\\
\leq &C \varepsilon(\|f_i-\lan{f_i}\ran\|_{L^\infty(\Torus)}^2+\lan f_i\ran ^2)+C(\varepsilon^{-1})\kappa^{2j+1} \int_\Torus |f_i|^2|y-y_i^\star|^{2j}dy\\
\leq& C_{GN}\varepsilon\|f_i-\lan{f_i}\ran\|_{L^2(\Torus)}^{2-\frac{1}{\gamma}}\||\pa_y|^\gamma f_i\|_2^{\frac{1}{\gamma}}+C\varepsilon{|\Torus|}\|f_i\|_2^2+C(\varepsilon^{-1})\kappa^{2j+1} \int |f_i|^2|y-y_i^\star|^{2j}dy\\
\leq&\frac{1}{2}\kappa\|f_i\|_{L^2(\Torus)}^2+C_{GN}\kappa^{1-2\gamma}\||\pa_y|^\gamma f_i\|_{L^2(\Torus)}^2+C\varepsilon {|\Torus|}\|f_i\|_2^2+C(\varepsilon^{-1})\kappa^{2j+1}\|f_i|y-y_i^\star|^j\|_{L^2(\Torus)}^2. 
\end{align*} 
Recalling that $\kappa>1$, choosing $\varepsilon$ small enough and reorganizing the terms yield that
\begin{align}
\kappa^{-2j}\|f_i\|_{L^2(\mathbb{T})}^2\leq& C\kappa^{-2j-2\gamma}\||\pa_y|^\gamma f_i\|_{L^2(\mathbb{T})}^2+C\||y-y_i^\star|^jf_i\|_{L^2(\Torus)}^2.
\end{align}
Now we set $\sigma=\kappa^{-2j-2\gamma}\leq 1$, then $\kappa^{-2j}=(\kappa^{-2j-2\gamma})^{\frac{2j}{2j+2\gamma}}=\sigma^{\frac{j}{j+\gamma}}$. 
As a result, we have derived the spectral gap estimate \eqref{Spectral_gap_sigma}. This concludes the proof.
\end{proof}



\ifx
The following spectral gap estimate in \cite{BCZ15} (ARMA version Remark 2.8) is crucial for the analysis
\begin{align}
\sigma ^{\frac{j_i}{j_i+1}}\|f\xi_i\|_2^2\leq& C\sigma\|\pa_y (f\xi_i)\|_2^2+C\| u' f\xi_i\|_2^2, \label{spectral_gap}
\end{align}
where $C$ is a universal constant. 
(Here $\xi_i$ is similar to the $\sqrt{\widetilde{\phi_i}}$ in the paper \cite{BCZ15}.) The small parameter $\sigma$ is any positive constant small enough. 
\fi

Next we consider critical points $\{y_i^\star\}_{i\in  \mathcal{I}_{\mathrm{near}}}$ in case a) \eqref{Case a}. 
\begin{lem} \label{lem:case a}
Assume the conditions in Theorem \ref{thm:semigroup_est_PS}. Assume that both of the following conditions hold: 

\noindent
a) the parameter $\delta\in(0,1)$ is small, i.e., 
\begin{align}\label{Choice_of_delta}
0< \delta(u)\leq \frac{1}{9}C_1^2(u)C_{\mathrm{spec}}(u),
\end{align}where $C_1$ is defined in \eqref{relation:u'_and_y-y_i_star_j} and $C_{\mathrm{spec}}$ is defined in \eqref{Spectral_gap}; 

\noindent
b) the threshold $\nu_0>0$ is smaller than a constant depending only on $u$ and $\delta$. 

\noindent
Then the estimate \eqref{Goal_i_by_i} holds for $i\in \mathcal{I}_{\mathrm{near}}$ \eqref{Case a}.
\end{lem}
\begin{proof} We focus on one critical point $y_i^\star$ and drop the subscript $i$ in the vanishing order $j_i$.  There are three main steps. In the first step, we introduce suitable cut-off functions and estimate their Sobolev norms. In the second and third step, we carry out the main estimates of the $L^2$-norm. Throughout the proof, we will choose the viscosity threshold $\nu_0(u,\delta)$ small in several occasions, and the final viscosity threshold will be chosen as the minimum of all. 

\noindent
\textbf{Step \#1: }Cut-off functions and their Sobolev norms. 
\ifx
By choosing $\sigma=(\nu|k|^{-1})^{2\frac{j_i+1}{j_i+1+2\gamma}}$ in \eqref{spectral_gap}, we end up with the following spectral gap estimate for $j\geq 1$,\begin{align}
\nu^{\frac{j_i+1}{j_i+1+2\gamma}} ^{\frac{2\gamma}{j_i+1+2\gamma}}\|f\xi_i\|_2^2\leq C\nu \|\pa_y(f\xi_i )\|_2^2+C\nu^{\frac{1-j_i}{j_i+1+2\gamma}} ^{\frac{2(j_i+1)}{j_i+1+2\gamma}}\|u' f\xi_i\|_2^2.\label{spectral_gap_j} 
\end{align}
This spectral gap estimate appears as equation (3.4) in \cite{BCZ15}. To simplify the notation, we will use $j$ to denote the vanishing order $j_i$.
\fi

We define a smooth partition of unity on the domain $\Torus$, i.e., $1=\psi_i^{0}+\psi_i^{+}+\psi_i^{-}$. First we choose the viscosity threshold $\nu_0(u)$ small so that $4\sqrt{C_0}\nu_0^{\frac{1}{j+1+2\gamma}}\leq \frac{1}{20}r_i$, where $C_0$ and $r_i$ are defined in \eqref{non-degenerate_0} and \eqref{non-degenerate}. The function $\psi_i^0$ has the following properties:
\begin{align}
1)\quad& \mathrm{support}{(\psi_i^0)}=B(y_i^\star; 4\sqrt{C_0(u)}\delta^{\frac{1}{j+1}}\nu^{\frac{1}{j+1+2\gamma}});  \label{psi_0_property_1}\\
2)\quad& \psi_i^0(y)\equiv1,\quad \forall y\in B(y_i^\star;2\sqrt{C_0(u)}\delta^{\frac{1}{j+1}}\nu^{\frac{1}{j+1+2\gamma}}); \label{psi_0_property_2}\\
3) \quad&|u(y)-\lambda| \geq  \delta \nu^{\frac{j+1}{j+1+2\gamma}} ,\quad \forall y\in \mathrm{support}(\xi_i)\backslash B(y_i^\star;  2\sqrt{C_0(u)}\delta^{\frac{1}{j+1}}\nu^{\frac{1}{j+1+2\gamma}}). \label{psi_0_property_3}
\end{align}
Here $C_0(u)\geq 1$ is defined in \eqref{non-degenerate_0}. 
We check  the last property as follows. By the condition \eqref{non-degenerate_0}, the assumption \eqref{Case a},  and  $j=j_i\geq1$, we observe that for $y\in \mathrm{support}(\xi_i)\backslash B(y_i^\star;  2\sqrt{C_0(u)}\delta^{\frac{1}{j+1}}\nu^{\frac{1}{j+1+2\gamma}})$,
\begin{align}
|\lambda-u(y)|\geq& \|u(y)-u_i^\star|-|\lambda-u_i^\star\|\geq \frac{1}{C_0(u)}|y-y_i^\star|^{j+1}-\delta\nu^{\frac{j+1}{j+1+2\gamma}}\\
\geq &\left({2^{j+1}C_0 ^{\frac{j+1}{2}}} C_0 ^{-1}-1\right)\delta\nu^{\frac{j+1}{j+1+2\gamma}}\geq  \delta\nu^{\frac{j+1}{j+1+2\gamma}}.
\end{align} 
The supports of the functions $\psi_i^+ $ and $\psi_i^- $ are adjacent to the support of  $\psi_i^0$. Since the support of the functions $\psi_i^\pm \xi_i^2$ are included in $\mathrm{support}(\xi_i)\backslash B(y_i^\star;  2\sqrt{C_0(u)}\delta^{\frac{1}{j+1}}\nu^{\frac{1}{j+1+2\gamma}})$, previous argument yields that \begin{align}
|u(y)-\lambda|\geq  \delta \nu^{\frac{j+1}{j+1+2\gamma}} ,\quad \forall y\in\mathrm{support}(\psi_i^\pm \xi_i^2).\label{psi_pm_property}
\end{align} 

Next we estimate the Sobolev norms. Since all three functions transition on interval of size $\mathcal{O}(\delta^{\frac{1}{j+1}}\nu^{\frac{1}{j+1+2\gamma}})$, their $\dot H^1,\dot H^2$-seminorms are bounded as follows
\begin{align}\label{psi_i_bound}&&& \|\pa_y \psi_i^s\|_2\leq C(u)\delta^{-\frac{1}{2(j+1)}}\nu^{-\frac{1}{2(j+1+2\gamma)}} ,\quad \|\pa_y^2 \psi_i^s\|_2\leq C(u)\delta^{-\frac{3}{2(j+1)}}\nu^{-\frac{3}{2(j+1+2\gamma)}} ,\quad\forall s\in\{0,+,-\}.
\end{align}
To estimate the $\dot H^\gamma$-seminorm for $\gamma\in(1/2,1)$, the $\dot H^{1/2}$-norm of the partition functions $\psi_i^s$ are required:
\begin{align}\|\psi_i^s\|_{\dot H^{1/2}}\leq C(u,\delta^{-1}, j)|\log (\nu  )|,\quad s\in\{0,+,-\}.\label{psi_i_bound_H1/2}
\end{align} 
The explicit estimate is as follows. Denote $\Lambda:=2\sqrt{C_0}\delta^{\frac{1}{j+1}} \nu^{\frac{1}{j+1+2\gamma}} $. Recall the Fourier characterization of the $\dot H^{\frac{1}{2}}$-seminorm: $\|\psi_i^s-\lan{\psi_i^s}\ran\|_{\dot H^{1/2}}^2=C\sum_{\ell\neq 0}|\wh \psi_i^s(\ell)|^2|\ell|$, where $\wh \psi_i^s(\ell)=\frac{1}{2\pi}\int_{-\pi}^\pi \psi_i^s(y) e^{- i \ell y}dy$.  
  Now we have the following two relations which are consequences of the integration by parts and the relation $e^{- i \ell y}=-\frac{1}{  i \ell}\frac{d}{dy}e^{-  i \ell y}$: 
\begin{align}
|\wh \psi^s_i(\ell)|\leq& \frac{1}{|\ell|}\int_{\Torus} |\pa_y \psi_i^s|dy\leq C\frac{1}{|\ell|},\, \ell\neq 0;\\
|\wh\psi^s_i(\ell)|\leq &\frac{1}{|\ell|^2}\int_{\Torus}|\pa_{yy}\psi_i^s|dy \leq \frac{1}{|\ell|^2}\Lambda^{-1},\, \ell\neq 0.
\end{align}
Now we estimate the $\dot H^{1/2}$-seminorm as follows
\begin{align*}
\|\psi_i^s\|_{\dot H^{1/2}}^2=&C\left(\sum_{0<|\ell|<\Lambda^{-1}}+\sum_{ |\ell|\geq\Lambda^{-1}}\right)|\wh \psi_i^s(\ell)|^2|\ell|
\leq C\sum_{0<|\ell|< \Lambda^{-1}}\frac{1}{|\ell|}+C\sum_{|\ell|\geq\Lambda^{-1}}\frac{\Lambda^{-2}}{|\ell|^3}\leq C\log\Lambda^{-1}+C.
\end{align*}
As a result, we have that $\|\psi_i^s\|_{\dot H^{1/2}}\leq C(u,\delta^{-1},j)|\log ( \nu )|$.

\ifx
\myb{Can  we  apply the Toland formula?:
\begin{align}
\|u\|_{\dot H^{1/2}}^2=C\int_{-\pi}^\pi\int_{-\pi}^\pi\left(\frac{u(x)-u(y)}{\sin(\frac{x-y}{{2}})}\right)^2dxdy.
\end{align}Now only the place where $|x-y|\leq \eta$ play roles. And the difference is nonzero only near (distance$\approx \eta$) the transition threshold. Now in this region, we have that 
\begin{align}
\int_{(y_i^\star-4\eta, y_i^\star+4\eta)}\int_{(y_i^\star-4\eta, y_i^\star+4\eta)}\frac{|u(x)-u(y)|^2}{|\sin(\frac{x-y}{2})|^2}dxdy\leq C\frac{1}{\eta^2}\eta^2=C. 
\end{align}But this is not all. There is a transition region, and we are kind of integrating a $1/|x|$ across scales.  Even in this formula, I saw a log loss. }\fi

Finally, we apply Gagliardo-Nirenberg interpolation inequality to derive the $\dot H^\gamma$ semi-norm,
\begin{align}\|\psi_i^s\|_{\dot H^{\gamma}}\leq& C_{GN}\|\psi_i^s\|_{\dot H^{1/2}}^{2-2\gamma}\|\psi_i^s\|_{\dot H^1}^{2\gamma-1}\leq C(u,\delta^{-1},j)|\log \nu|^{2-2\gamma}\nu^{-\frac{2\gamma-1}{2(j+1+2\gamma)}},\quad\gamma\in(1/2,1);\\\|\psi_i^s\|_{\dot H^{\gamma}}\leq& C_{GN}\|\psi_i^s\|_{\dot H^1}^{2-\gamma}\|\psi_i^s\|_{\dot H^2}^{\gamma-1}\leq C(u,\delta^{-1},j)\nu^{-\frac{2\gamma-1}{2(j+1+2\gamma)}},\quad\gamma\in[1,2].&
\end{align}%
We combine these two estimates with the parameter $\al(\gamma)$ \eqref{al }:
\begin{align}\label{H_gamma_psi}
\|\psi_i^s\|_{\dot H^{\gamma}}\leq C(u,\delta^{-1},j_i)\nu^{-\frac{2\gamma-1}{2(j+1+2\gamma)}}|\log \nu|^{2\al(\gamma)},\quad \gamma\in(1/2,2].
\end{align}

We apply the Minkowski inequality to decompose the $L^2$-norm as follows
\begin{align}
\|w\xi_i\|_2^2\leq 3\|w {\psi_i^0}\xi_i\|_2^2+3\|w  {\psi_i^+}\xi_i\|_2^2+3\|w {\psi_i^-}\xi_i\|_2^2.\label{case_a_decomposition}
\end{align}
This concludes step \#1.

\noindent
\textbf{Step \#2: } Estimation of the $L^2$-norm $\|w\psi_i^0\xi_i\|_2^2$. 
We apply the following product rule for $f,g\in H^\gamma\cap L^\infty$:
\begin{align}\label{product_rule}
\||\pa_y|^\gamma (fg)\|_{L^2_y(\Torus)}\leq C\|f\|_{H_y^\gamma(\Torus)}\|g\|_{L^\infty_y(\Torus)}+C\|f\|_{L^\infty_y(\Torus)}\|g\|_{H_y^\gamma(\Torus)}
\end{align} The proof of the product rule on $\rr$ can be found in various textbooks (see, e.g., appendix of \cite{Tao06}.) and a small modification yields \eqref{product_rule}. Application of the spectral gap \eqref{Spectral_gap} and the product rule yields that  \myb{((A.17) of Terry Tao's book on dispersive equations ($\rr$). One can also check the lecture notes by S. Klainerman.)  
}
\begin{align}
\nu^{\frac{j+1}{j+1+2\gamma}}  \|w \psi_i^0\xi_i\|_2^2&
\leq C_{\mathrm{spec}}(u)\nu \||\pa_y|^\gamma(w \psi_i^0\xi_i)\|_2^2+C_{\mathrm{spec}}(u)\nu^{\frac{1-j}{j+1+2\gamma}} \|u'w \psi_i^0\xi_i\|_2^2\\
&\leq C(u)\nu(\||\pa_y|^\gamma w \|_2^2+\|w\|_2^2)\|\psi_i^0\xi_i\|_\infty^2+C(u)\nu\|w \|_\infty^2(\||\pa_y|^\gamma(\psi_i^0\xi_i)\|_2^2+\|\psi_i^0\xi_i\|_2^2)\\
&\quad+C_{\mathrm{spec}}(u)\nu^{\frac{1-j}{j+1+2\gamma}} \|u'w \psi_i^0\xi_i\|_2^2\\
&=:T_1+T_2+T_3.\label{Case_a_T123}
\end{align}Combining the fact that $\|\xi_i\|_\infty,\, \|\psi_i^0\|_\infty$ are bounded by $1$, and the estimate \eqref{Real}, we obtain that the first term is bounded, i.e., \begin{align}\label{Case_a_T_1}
T_1\leq C(u)\nu(\||\pa_y|^\gamma w\|_2^2+\|w\|_2^2)\leq C(u)\|F\|_2\|w\|_2+C(u)\nu\|w\|_2^2.
\end{align} To estimate the $T_2$ term, we first estimate the quantity $\||\pa_y|^\gamma(\psi_i^0\xi_i)\|_2$ with the product rule \eqref{product_rule}, the $\xi_i$-estimate \eqref{Est_xi_i}, the $L^\infty$-bounds $\|\xi_i\|_\infty+\|\psi_i^0\|_\infty\leq 2$, and the $\dot H^\gamma$-estimate of $\psi_i^s$ \eqref{H_gamma_psi} as follows
\begin{align}
\||\pa_y|^{\gamma}(\psi_i^0\xi_i)\|_2\leq& C(u,\|\xi_i\|_{H^2})(\|\psi_i^0\|_{\dot H^{\gamma}}+1)
\leq C(u,\delta^{-1}, j)|\log (\nu)|^{2\al(\gamma)}\nu^{\frac{1-2\gamma}{2(j+1+2\gamma)}} .\label{pay_gamma_psi_xi_L2}
\end{align}
Now we combine \eqref{pay_gamma_psi_xi_L2} with the $L^\infty$-bounds $\|\xi_i\|_\infty+\|\psi_i^0\|_\infty\leq 2$, and apply H\"older inequality, Young inequality and Gagliardo-Nirenberg interpolation inequality to derive the following 
\begin{align}
T_2\leq &C(u,\delta^{-1},j)\nu \left(\||\pa_y|^\gamma w \|_2^{\frac{1}{\gamma}}\|w \|_2^{2-\frac{1}{\gamma}}+\|w \|_2^2\right)\nu^{\frac{1-2\gamma}{j+1+2\gamma}}|\log( \nu)| ^{4\al(\gamma)} \\
\leq &C(B,u,\delta^{-1}, j)\nu|\log\nu|^{8\gamma\al(\gamma)}\||\pa_y|^\gamma w\|_2^2+\left(\frac{1}{B}\nu^{\frac{j+1}{j+1+2\gamma}}+C(u,\delta^{-1},j)\nu^{\frac{j+2}{j+1+2\gamma}}|\log\nu|^{4\al(\gamma)}\right)\|w\|_2^2
\end{align}
Now we  apply the relation \eqref{Real}  to obtain the following
\begin{align}\label{Case_a_T_2}
T_2
\leq C(B,u,\delta^{-1}, j)|\log\nu|^{8\gamma\al(\gamma)}\|F\|_2\|w\|_2+\nu^{\frac{j+1}{j+1+2\gamma}}\left(\frac{1}{B}+C(u,\delta^{-1},j)\nu^{\frac{1}{j+1+2\gamma}}|\log\nu|^{4\al(\gamma)}\right)\|w\|_2^2.
\end{align}
To estimate the last term $T_3$ in \eqref{Case_a_T123}, we apply the condition \eqref{non-degenerate} that $|u'(y)|\leq C_1(u) |y-y_i^\star|^{j}\leq C_1(u)\delta^{\frac{j}{j+1}}\nu^{\frac{j}{j+1+2\gamma}}$ on the support of $\psi_i^0\xi_i$ and choice of $\delta(u)$ \eqref{Choice_of_delta} to obtain 
\begin{align}
T_3\leq& C_{\mathrm{spec}}(u)C_1^2(u)\delta^{\frac{2j}{j+1}}\nu^{\frac{1+j}{j+1+2\gamma}}\|w\psi_i^0\xi_i\|_2^2\\
\leq &\left(C_{\mathrm{spec}}(u)C_1^2(u)\delta\right) \nu^{\frac{1+j}{j+1+2\gamma}}\|w\psi_i^0\xi_i\|_2^2\leq \frac{1}{9}\nu^{\frac{1+j}{j+1+2\gamma}}\|w\psi_i^0\xi_i\|_2^2.\label{Case_a_T_3}
\end{align}
\ifx 
and , and As a result, 
\begin{align}
\nu^{\frac{j+1}{j+1+2\gamma}}&\|w \psi_i^0\xi_i\|_2^2\\
\leq&C\nu\||\pa_y|^\gamma w \|_2^2\|\psi_i^0\xi_i\|_\infty^2+C(\delta^{-1})\nu|\log( \nu)| ^{4\al(\gamma)}\||\pa_y|^\gamma w \|_2^{\frac{1}{\gamma}}\|w \|_2^{2-\frac{1}{\gamma}}\nu^{\frac{1-2\gamma}{j+1+2\gamma}} \\
&+C(\delta^{-1})\nu\|w \|_2^2|\log( \nu)| ^{4\al(\gamma)}\nu^{\frac{1-2\gamma}{j+1+2\gamma}} + C\nu^{\frac{1-j}{j+1+2\gamma}} \delta^{\frac{2j}{j+1}}\nu^{\frac{2j}{j+1+2\gamma}} \|w \psi_i^0\xi_i\|_2^2.
\end{align}

Here the crucial point is that the constant in the last term does not depend on the parameter $\delta^{-1}$ and hence one can take the $\delta$ small enough such that the last term has small coefficients. Combining Young's inequality, the relation \eqref{Real}, and some calculations, we have that \fi 
Here the crucial point is that the coefficients in the bound of $T_3$ only depends on the shear profile. By taking $\delta$ to be small compared to $C_{\mathrm{spec}}(u)C_1^2(u)$, we have that the $T_3$ term can be absorbed by the left hand side of \eqref{Case_a_T123}. Therefore, by combining the estimates \eqref{Case_a_T123}, \eqref{Case_a_T_1}, \eqref{Case_a_T_2}, and \eqref{Case_a_T_3},   we obtain
\begin{align}
 \|w \psi_i^0\xi_i\|_2^2
\leq & C(B,u,\delta^{-1},j)  \nu^{-2\frac{j+1}{j+1+2\gamma}} |\log( \nu )| ^{16\gamma\al(\gamma)}\|F\|_2^2\\
&+\left(\frac{1}{B}  +C(u,\delta^{-1},j)\nu^{\frac{1}{j+1+2\gamma}}|\log \nu|^{4\al(\gamma)}\right) \|w \|_2^2.\label{psi_0_term}
\end{align}\ifx
As a result, we have that
\begin{align}
\|w \psi_i^0\xi_i\|_2^2\leq C(\delta^{-1}) \nu^{-2\frac{j+1}{j+1+2\gamma}} |\log( \nu )| ^{16\gamma\al(\gamma)}\|F\|_2^2+C\delta^{\frac{2j}{j+1}}\|w \|_2^2.
\end{align} 
\fi This concludes the step \#2.

\noindent 
\textbf{Step \#3:} Estimation of the $\|w \psi_i^\pm\|_2^2$ terms in \eqref{case_a_decomposition}. Since $|u(y)-\lambda|$ has quantitative positive lower bound \eqref{psi_pm_property} on the domain of integration, we apply the relation \eqref{Imaginary_with_f}, H\"older inequality  and the product rule \eqref{product_rule} to obtain  the following estimate, 
\begin{align}
\|w \psi_i^\pm\xi_i\|_2^2=&\bigg|\int_{\Torus}|w |^2\frac{(u-\lambda)}{(u-\lambda)}|\psi_i^\pm|^2\xi_i^2dy\bigg|\leq \delta^{-1}\nu^{-\frac{j+1}{j+1+2\gamma}} \bigg|\int_\Torus |w |^2(u-\lambda)|\psi_i^\pm|^2\xi_i^2 dy\bigg|\\
\leq& C\delta^{-1}\nu^{-\frac{j+1}{j+1+2\gamma}}  \bigg(\|F\|_2\|w \|_2\|\psi_i^\pm\xi_i\|_\infty^2+\nu \||\pa_y|^\gamma w \|_2(\||\pa_y|^\gamma w \|_2 +\|w\|_2 )\|\psi_i^\pm\xi_i\|_\infty\\
&\quad\quad \quad\quad\quad \quad\quad+\nu \|  |\pa_y|^\gamma w\|_2 \|w\|_\infty(\||\pa_y|^\gamma (\psi_i^\pm\xi_i)^2\|_2+\|(\psi_i^\pm\xi_i)^2\|_2)\bigg)\\
=&:T_4+T_5+T_6.\label{Case_a_T456}
\end{align}
The $T_4$ can be estimated using the relation \eqref{Real}, the fact that $\|\xi_i\|_\infty,\,\|\psi_i^s\|_\infty\leq 1$, and the Young inequality as follows
\begin{align}
T_4\leq C(B,\delta^{-1})\nu^{-2\frac{j+1}{j+1+2\gamma}}\|F\|_2^2+\frac{1}{B}\|w\|_2^2.\label{Case_a_T_4}
\end{align} 
Next we estimate $T_5$ term in \eqref{Case_a_T456}. Application of the relation \eqref{Real}, H\"older inequality and Young inequality yields that 
\begin{align}
T_5\leq C(B,\delta^{-1}, j)\nu^{-2\frac{j+1}{j+1+2\gamma}}\|F\|_2^2+\left(\frac{1}{B}+C(\delta^{-1},j)\nu^{\frac{2\gamma}{j+1+2\gamma}}\right)\|w\|_2^2.\label{Case_a_T_5}
\end{align}
To estimate term $T_6$ in \eqref{Case_a_T456}, 
we recall the quantitative estimates of $\psi_i$  \eqref{psi_i_bound} and $\xi_i-$estimate \eqref{Est_xi_i}, and apply a similar argument to \eqref{pay_gamma_psi_xi_L2} to obtain  that 
\begin{align}
\||\pa_y|^\gamma(\psi_i^\pm\xi_i)^2\|_2
\leq &C(u,\delta^{- 1}, j)\nu^{\frac{1-2\gamma}{2(j+1+2\gamma)}} |\log (\nu  )|^{2\al(\gamma)}.
\end{align}
Now we apply the relation \eqref{Real}, Young inequality, Gagliardo-Nirenberg interpolation inequality to  obtain
\begin{align}
T_6\leq &C\delta^{-1}\nu^{-\frac{j+1}{j+1+2\gamma}}\nu \left(\||\pa_y|^\gamma w \|_2^{1+\frac{1}{2\gamma}}\|w \|_2^{1-\frac{1}{2\gamma}}+\||\pa_y|^\gamma w \|_2\|w \|_2\right)(\||\pa_y|^\gamma(\psi_i^+\xi_i)^2\|_2+1)\\
\leq&  C(B,\delta^{-1},j)\left(\nu^{-\frac{j+1}{j+1+2\gamma}} |\log (\nu )|
^{\frac{8\gamma\al(\gamma)}{2\gamma+1}} +\nu^{-\frac{j}{j+1+2\gamma}}|\log \nu|^{4\al(\gamma)} \right) \nu\||\pa_y|^\gamma w \|_2^2+\frac{1}{B}\|w \|_2^2\\
\leq& C(B,\delta^{-1},j)\nu^{-2\frac{j+1}{j+1+2\gamma}} |\log (\nu )|
^{16\gamma\al(\gamma)}\|F\|_2^2+\frac{1}{B}\|w\|_2^2.\label{Case_a_T_6}
\end{align}
 Combining the estimates \eqref{Case_a_T456}, \eqref{Case_a_T_4}, \eqref{Case_a_T_5} and \eqref{Case_a_T_6} , we obtain that 
\begin{align}
\|w& \psi_i^\pm\xi_i\|_2^2\\
\leq&C(B,u,\delta^{-1},j) \nu^{-2\frac{j+1}{j+1+2\gamma}} |\log \nu|
^{{16\gamma \al(\gamma)}}  \|F\|_2^2+\left(\frac{1}{B}+C(u,\delta^{-1},j_i) \nu^{\frac{1}{j+1+2\gamma}	}|\log \nu|^{4\al(\gamma)}\right)\|w \|_2^2,\label{psi_pm_term}
\end{align} 
for any $B>1$. Now combining the decomposition  \eqref{case_a_decomposition}, the estimates \eqref{psi_0_term} and \eqref{psi_pm_term}, we obtain the estimate \eqref{Goal_i_by_i} for $i\in \mathcal{I}_{\mathrm{near}}$ \eqref{Case a}.\ifx
\begin{align}
\|w \xi_i\|_2^2\leq& 4(\|w \psi_i^0\xi_i\|_2^2+\|w \psi_i^+\xi_i\|_2^2+\|w \psi_i^-\xi_i\|_2^2)\\
\leq &C(B,\delta^{-1}) \nu^{-2\frac{j+1}{j+1+2\gamma}} |\log(\nu  )|^{16\gamma\al(\gamma)}\|F\|_2^2+(\frac{1}{B}+C\delta^{\frac{2j}{j+1}})\|w \|_2^2.
\end{align} 
This is consistent with  and hence the treatment of case a) is completed.\fi \end{proof}
Next we treat the case where $i\in \mathcal{I}_{\mathrm{interm}}$  \eqref{Case b}. 

\begin{lem} \label{lem:case b}
Assume the condition in Theorem \ref{thm:semigroup_est_PS}. For $i\in\mathcal{I}_{\mathrm{interm}}$ \eqref{Case b} and for $\nu_0(u)$ small enough, the estimate  \eqref{Goal_i_by_i} holds. 
\end{lem}
\begin{proof}
Here we drop the subscript $i$ in the vanishing order $j_i$.  We organize the proof into three steps. 

\noindent 
\textbf{Step \#1:} Before estimating the $L^2$-norm $\|w \xi_i\|_{L^2(\Torus)}$,  several definitions are introduced. First, we consider the set $E_i$ and its compliment $E_i^c$ associated with each critical point:
\begin{align}
E_i:=\left\{y\bigg\|u(y)-\lambda|\leq \frac{1}{9}\delta  \nu^{\frac{j+1}{j+1+2\gamma}},\,y\in \mathrm{support}(\xi_i^2)\right\}, \quad E_i^c=\Torus\backslash E_i.\label{defn:E} 
\end{align} 
Next, we define functions $\phi_i$ which change sign near the spectral parameter $\lambda$, and have transition layer adapted to the set $E_i$. 
Given the shear profile $u(y),$ we first define the $y_{i;\lambda}$ such that $u(y_{i;\lambda})=\lambda,\, y_{i;\lambda}\in B(y_i^\star,4r_i), \,  \mathrm{support} (\xi_i^2)=B(y_i^\star,2r_i)$. If the vanishing  order $j$  of the critical point is odd, then there are two points $y_{i;\lambda }$ such that $u(y_{i;\lambda})=\lambda$. In this case, we use $y_{i;\lambda}^-$ and $y_{i;\lambda}^+$ to represent them and  use $y_{i;\lambda}$ to denote the set $\{y_{i;\lambda}^+,y_{i;\lambda}^-\}$. 
Since the balls $\{B(y_i^\star, 4r_i)\}_{i=1}^N$ are mutually disjoint, the $y_{i;\lambda}$ associated with different critical points are distinct.  We define the function $\{\phi_i\}_{i=1}^N\in (C^\infty(\mathbb{T}))^N$: 
\begin{align}
a) &\quad\phi_i =\left\{\begin{array}{ccc}\mathrm{sign}(u(y)-\lambda),\quad& |u(y)-\lambda|\geq \frac{1}{9}\delta   \nu^{\frac{j+1}{j+1+2\gamma}},\quad \text{dist}(y,y_i^\star)< 4r_i;\\
\text{monotone},&\quad|u(y)-\lambda|\leq \frac{1}{9} \delta \nu^{\frac{j+1}{j+1+2\gamma}},\quad \text{dist}(y,y_i^\star)< 4r_i;\\
\text{smooth},&\quad \mathrm{dist}\left( y,y_i^\star\right)\geq 4r_i=\mathrm{diameter}(\text{support}(\xi_i));
\end{array}\right.\\
b) &\quad \|\pa_y \phi_i\|_{L^2}\leq C \delta^{-\frac{1}{2(j+1)}} \nu^{-\frac{1}{2(j+1+2\gamma)}},\quad \|\pa_y^2 \phi_i\|_{L^2}\leq C \delta^{-\frac{3}{2(j+1)}} \nu^{-\frac{3}{2(j+1+2\gamma)}};\\
&\quad \||\pa_y|^{{1/2 }} \phi_i\|_{L^2}\leq C(\delta^{-1},j)|\log(\nu )|,\quad
\||\pa_y|^\gamma\phi_i\|_{L^2} \leq C(\delta^{-1}, j)|\log (\nu)|^{2\al(\gamma)}\nu^{\frac{1-2\gamma}{2(j+1+2\gamma)}};\\
c) &\quad\phi_i(y_{i;\lambda})=\phi_i(y_{i;\lambda}^\pm)=0,\quad\|\phi_i\|_{L^\infty}\leq 1;\\
d)&\quad \phi_i(y)=\phi_i(y_{i;\lambda})+\phi_i'(y_{i;\lambda})(y-y_{i;\lambda}), \quad |y-y_{i;\lambda}|\leq\frac{1}{L(u )}\delta^{\frac{1}{j+1}} \nu^{\frac{1}{j+1+2\gamma}},\quad \text{if } j \text{ is even};\\
&\quad \phi_i(y)=\phi_i(y_{i;\lambda}^\pm)+\phi_i'(y_{i;\lambda}^\pm)(y-y_{i;\lambda}^\pm), \quad |y-y_{i;\lambda}^\pm|\leq\frac{1}{L(u )}\delta^{\frac{1}{j+1}} \nu^{\frac{1}{j+1+2\gamma}},\quad\text{if } j \text{ is odd}.
\end{align} 
Since the function $\phi_i$ is very similar to the function $\psi_i^s$ in the proof of Lemma \ref{lem:case a}, the regularity estimates are similar to those of $\psi_i^s$. 
Here the universal constant $L(u )$ is chosen such that if $|y-y_{i;\lambda}^\pm|\leq\frac{1}{L(u )}\delta^{\frac{1}{j+1}} \nu^{\frac{1}{j+1+2\gamma}}$, then $|u(y)-\lambda|\leq \frac{1}{36}\delta\nu^{\frac{j+1}{j+1+2\gamma}} $. By the estimate of $u$ on the support of $\xi_i$ \eqref{non-degenerate_0},  the existence of this universal constant $L(u )$, which is independent of $\nu $, is guaranteed.   

\begin{figure}
\includegraphics[scale=0.451]{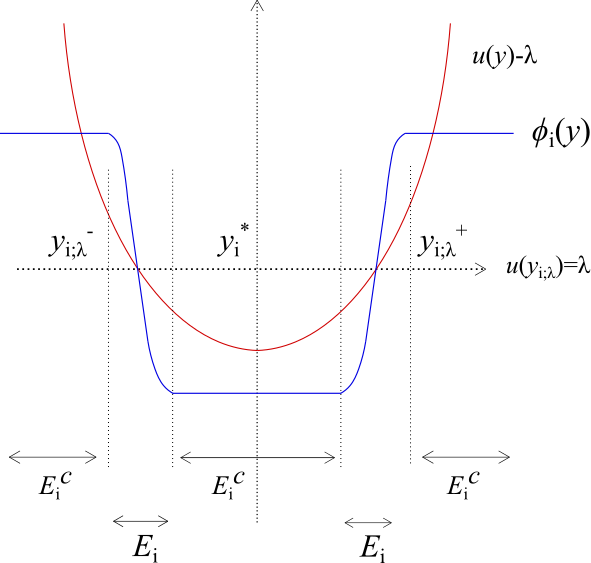}\quad
\includegraphics[scale=0.451]{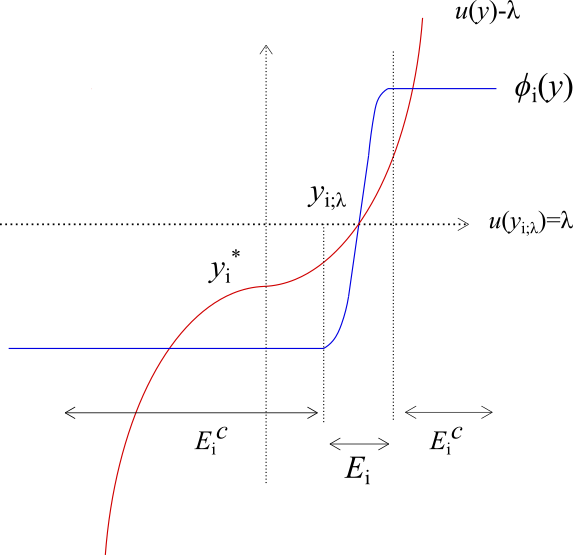}
\centering
\caption{Function $\phi_i(y)$ for $j$ odd(left).  Function $\phi_i(y)$ for $j$ even (right).} 
\end{figure}

Now we estimate the area of the set $ E_i$. Based on the relative position between $\lambda$ and $u_i^\star$, we distinguish between two cases: 
\begin{align}(b_1) \quad |y_{i;\lambda}-y_i^\star|\leq r_i;\quad  (b_2)\quad r_i<|y_{i;\lambda}-y_i^\star|\leq 3r_i.\label{b12}
\end{align}
In case $(b_1)$ \eqref{b12}, we first consider the case where the vanishing order $j$ is even. As a result, the shear profile is strictly monotone near the critical point and there is a unique $y_{i;\lambda}$ such that $u(y_{i;\lambda})=\lambda$.  We consider arbitrary $y\in  E_i$ \eqref{defn:E}.  Note that in this case, the product $(y-y_i^\star)(y_{i;\lambda}-y_i^\star)>0$. Application of the mean value theorem, the $u$-estimates \eqref{non-degenerate_0}, \eqref{non-degenerate} on the support of $\xi_i$  and definition  of $E_i$ \eqref{defn:E} yields that 
\begin{align}
|y-y_{i;\lambda}|\leq &\frac{|u(y)-u(y_{i;\lambda})|}{\min_{z\in [y,y_{i;\lambda}]\cup[y_{i;\lambda},y]}|u'(z)|}
\leq C_1(u)\frac{|u(y)-\lambda|}{\min\{|y-y_i^\star|^j,|y_{i;\lambda}-y_i^\star|^j\}}\\\leq&\frac{C(C_1(u),C_2(u))\delta\nu^{\frac{j+1}{j+1+2\gamma}} }{\min\{|u(y)-u_i^\star|^{\frac{j}{j+1}},|\lambda-u_i^\star|^{\frac{j}{j+1}}\}}\leq C(u)\delta^{\frac{1}{j+1}} \nu^{\frac{1}{j+1+2\gamma}}.\label{Choice_of_K}
\end{align}
Since the estimate holds for all $y\in E_i$, we have that \begin{align}\label{Area_E_i_cap_xi_i}{|E_i|}\leq C(u)\delta^{\frac{1}{j+1}} \nu^{\frac{1}{j+1+2\gamma}}.
\end{align} If the vanishing order $j$ is odd, then $y_{i;\lambda}=\{y_{i;\lambda}^-,y_{i,\lambda}^+\}$. Now we define  $E_i^\pm=E\cap \{y|(y-y_i^\star)(y_{i;\lambda}^\pm-y_i^\star)>0\}$. 
Now carrying out similar argument yields \eqref{Area_E_i_cap_xi_i}.  \myb{The remaining argument is similar. Check?!}  
In the case $(b_2)$ \eqref{b12}, by picking $\nu_0$ small enough in the definition  of the set $E_i$ \eqref{defn:E}, we have that for $y\in E_i$, the following the derivative lower bound of $u$ holds,
\begin{align}
\min_{z\in [y,y_{i;\lambda}]\cup[y_{i;\lambda},y]}|u'(z)|\geq \frac{1}{C(u)}>0.
\end{align} 
Combining this derivative $|u'|$ lower bound and a similar argument to \eqref{Choice_of_K} yields that ${|E_i|}\leq C(u)\delta^{\frac{1} {j+1}}\nu^{\frac{1}{j+1+2\gamma}} $.
 
To conclude step \#1, we estimate a specific function which will be applied later. We first consider the case where $j$ is even. We apply the properties of the functions $\phi_i,\, \xi_i$ and the mean value theorem to estimate the following function for $\forall y\in \Torus$,
\begin{align}
\xi_i^2(y) \frac{ |\phi_i (y)|}{ |u(y)-\lambda| } \leq& \xi_i^2(y)\mathbf{1}_{E_i}(y) \frac{\|\pa_y\phi_i\|_\infty |y-y_{i;\lambda}|}{|u(y)-u(y_{i;\lambda})|} +\xi_i^2(y)\mathbf{1}_{ E_i^c}(y) 9\delta^{-1}\nu^{-\frac{j+1}{j+1+2\gamma}} \\
\leq & \xi_i^2(y)\mathbf{1}_{E_i}(y)\|\pa_y\phi_i\|_\infty   \frac{1}{\min_{z\in E_i}|u'(z)|} +\xi_i^2(y)\mathbf{1}_{ E_i^c}(y) 9\delta^{-1}\nu^{-\frac{j+1}{j+1+2\gamma}}. \label{factor_T1_0}
\end{align}
Now by the behavior of the shear profile $u$ on the support of $\xi_i$ \eqref{non-degenerate_0}, \eqref{non-degenerate}, the $\lambda$ constraint \eqref{Case b}, the definition of $E_i$ \eqref{defn:E}, we observe that  
\begin{align}
\min_{y\in E_i}|u'(y)|\geq& C_1^{-1}(u)\min_{y\in E_i}|y-y_i^\star|^{j}\geq C_1^{-1}C_0^{-1}\min_{y\in E_i}|u(y)-u(y_i^\star)|^{\frac{j}{j+1}} \\
\geq& C_1^{-1}C_0^{-1}\min_{y\in E_i}\|u_i^\star-\lambda|-|u(y)-\lambda\|^{\frac{j}{j+1}} \geq C_1^{-1}C_0^{-1}\bigg|\delta \nu^{\frac{j+1}{j+1+2\gamma}}-\frac{1}{9}\delta \nu^{\frac{j+1}{j+1+2\gamma}}\bigg|^{\frac{j}{j+1}}\\
\geq &C(u)^{-1}\delta^{\frac{j}{j+1}}\nu^{\frac{j}{j+1+2\gamma}}.
\end{align}
Combining the lower bound, the estimate \eqref{factor_T1_0} and the property b) of the function $\phi_i$, we have obtained the following 
\begin{align}\xi_i^2(y) \frac{| \phi_i (y)|}{ |u(y)-\lambda| } \leq& C\delta^{-1}\nu^{-\frac{j+1}{j+1+2\gamma}},\quad \forall y\in \Torus.\label{factor_T1} 
\end{align}
Here if $j$ is odd, then we can replace the $y_{i;\lambda}$ above by either $y_{i;\lambda}^+$ or $y_{i;\lambda}^-$. Moreover, we will decompose the domain $\Torus$ into the three subdomains $E_i^c$, $E_i^\pm$, where $E_i^\pm$ denotes the connect component of $E_i$ which contains $y_{i;\lambda}^\pm$. Application of similar estimate above yields the estimate \eqref{factor_T1}.  This concludes step \#1. \myb{Check!}

\noindent
\textbf{Step \#2: }Estimation of $\|w \xi_i\|_{L^2(E_i^c)}$.  Direct application of the constraint \eqref{defn:E} and property a) of the function $\phi_i$ yields that
\begin{align}\label{wk_xi_L2_E_i_c}
\|w \xi_i\|_{L^2(E_i^c)}^2=&\int_{E_i^c}\xi_i^2\frac{|u-\lambda \|\phi_i|}{|u-\lambda\|\phi_i|}|w |^2dy\leq C \delta^{-1}\nu^{-\frac{j+1}{j+1+2\gamma}}\int_{\Torus} |u-\lambda|\xi_i^2 |\phi_i\|w |^2dy.
\end{align}
\ifx
\textcolor{blue}{For checking purpose: for fixed $y$, pick the point $y_{i;\lambda}^\pm$ in the set $y_{i;\lambda}$ (Recall that if $j$ is odd, $y_{i;\lambda}$ is a set containing two points, and if $j$ is even, $y_{i;\lambda}$ is a point.) which is closest to the point $y$ so that $(y-y_i^\star)$ and $(y_{i;\lambda}-y_i^\star)$. It is enough to estimate $|u(y)-u(y_{i;\lambda})|=|u-\lambda|$, and applying the Taylor expansion (with $K$ \eqref{defn:E} and $r_i$ chosen small enough)
\begin{align}
|u(y)-\lambda|=&|(u(y)-u_i^\star)-(u(y_{i;\lambda})-u_i^\star)|\\
=&\bigg|\frac{u^{(j+1)}(y_i^\star)}{(j+1)!}((y-y_i^\star)^{j+1}-(y_{i;\lambda}-y_i^\star)^{j+1})+O(|y-y_i^\star|^{j+2},|y_{i;\lambda}-y_i^\star|^{j+2})\bigg|\\
\geq &|\frac{1}{C}(y-y_{i;\lambda})\sum_{\ell=0}^j(y-y_i^\star)^{j-\ell}(y_{i;\lambda}-y_i^\star)^\ell|-O(|y-y_i^\star|^{j+2},|y_{i;\lambda}-y_i^\star|^{j+2})\\
\geq &\frac{1}{C}|y-y_{i;\lambda}|\min\{|y-y_i^\star|^j,|y_{i;\lambda}-y_i^\star|^j\}-O(|y-y_i^\star|^{j+2},|y_{i;\lambda}-y_i^\star|^{j+2})\\
\geq&\frac{1}C|y-y_{i;\lambda}\|y_{i;\lambda}-y_i^\star|^j-\frac{C}{K^j}|y-y_{i;\lambda}|^{j+1}-
\end{align}}
\fi
Now we test the equation \eqref{Resolvent} by $\overline{w }\phi_i\xi_i^2$, 
\begin{align}
\int &|u-\lambda\|\phi_i|\xi_i^2|w |^2dy=\int(u-\lambda)\phi_i\xi_i^2|w |^2dy
=\mathrm{Im}\int_\Torus F\overline{w}\phi_i\xi_i^2  dy-\nu\mathrm{Im} \int_\Torus |\pa_y|^\gamma w |\pa_y|^\gamma (\phi_i\xi_i^2\overline{w})dy.\label{Case_b_proof_1}
\end{align}
Now we apply the bound \eqref{factor_T1}, property a) and c) of the functions $\phi_i$ and the product estimate \eqref{product_rule}, H\"older inequality, Young's inequality and Gagliardo-Nirenberg interpolation inequality to  get the following 
\begin{align}
\int |u-\lambda|&|\phi_i|\xi_i^2|w |^2dy\\
\leq& {C}{ }\bigg(\norm{\frac{F\sqrt{|\phi_i|}\xi_i}{\sqrt{|u-\lambda|}}}_{2}\|\sqrt{|u-\lambda|}\sqrt{|\phi_i|}\xi_i w \|_2+\nu \||\pa_y|^\gamma w \|_2(\||\pa_y|^\gamma w \|_2 +\|w\|_2)\|\phi_i\xi_i^2\|_\infty\\
&\quad\quad+\nu\||\pa_y |^\gamma w \|_2\|w \|_\infty(\|\phi_i\xi_i^2\|_{\dot H^\gamma}+\|\phi_i\xi_i^2\|_{L^2}\bigg)\\
\leq& {C(\|\xi_i\|_{ W^{1,\infty}})}{ }\bigg(\norm{\frac{F\sqrt{|\phi_i|}\xi_i}{\sqrt{|u-\lambda|}}}_{ 2}\|\sqrt{|u-\lambda|}\sqrt{|\phi_i|} \xi_iw \|_2+\nu \||\pa_y |^\gamma w \|_2(\||\partial_y|^\gamma w\|_2+\|w\|_2) \\
&\quad\quad\quad\quad\quad\quad+\nu(\||\pa_y|^\gamma w \|_2^{1+\frac{1}{2\gamma}}\|w \|_2^{1-\frac{1}{2\gamma}}+\||\pa_y|
^\gamma w \|_2\|w \|_2)(\||\pa_y|^\gamma\phi_i\|_{2}+1)\bigg)\\ 
\leq & {C(u)}{ }\norm{\frac{F\sqrt{|\phi_i|}\xi_i}{\sqrt{|u-\lambda|}}}_{ 2}^2+\frac{1}{2}\|\sqrt{|u-\lambda|}\sqrt{|\phi_i|} \xi_iw \|_2^2+ {C(u)}\nu \||\pa_y|^\gamma w \|_2(\||\pa_y|^\gamma w \|_2+\|w\|_2) \\
& +C(u)\nu(\||\pa_y|^\gamma w \|_2^{1+\frac{1}{2\gamma}}\|w \|_2^{1-\frac{1}{2\gamma}}+\||\pa_y|^\gamma w \|_2\|w \|_2)( \|\phi_i\|_{\dot H^\gamma} +1).
\end{align}
Note that the second term on the right hand side gets absorbed by the left hand side, which implies the following estimate
\begin{align}
\int &|u-\lambda\|\phi_i|\xi^2_i|w |^2dy\leq {C}(u) \norm{\frac{F\sqrt{|\phi_i|}\xi_i}{\sqrt{|u-\lambda|}}}_{ 2}^2+ {C}(u) \nu\left( \||\pa_y|^\gamma w \|_2^2+ \||\pa_y|^\gamma w \|_2\|w\|_2\right)\\
&+C(u)\nu\left(\||\pa_y|^\gamma w \|_2^{1+\frac{1}{2\gamma}}\|w \|_2^{1-\frac{1}{2\gamma}}+\||\pa_y|^\gamma w \|_2\|w \|_2\right)( \|\phi_i\|_{\dot H^\gamma} +1).
\end{align}
Combining this relation with \eqref{wk_xi_L2_E_i_c} and the properties of the function $\phi_i$ yields the following 
\begin{align}
\|w \xi_i\|_{L^2(E_i^c)}^2\leq&C(u,\delta^{-1}) \nu^{-\frac{j+1}{j+1+2\gamma}}\bigg( {}{ }\norm{\frac{F\sqrt{|\phi_i|}\xi_i}{\sqrt{|u-\lambda|}}}_{L^2}^2+ {\nu } (\||\pa_y|^\gamma w \|_2^2+\||\pa_y|^\gamma w\|_2\|w\|_2)\\
&+  {\nu} (\||\pa_y|^\gamma w \|_2^{1+\frac{1}{2\gamma}}\|w \|_2^{1-\frac{1}{2\gamma}}+\||\pa_y|^\gamma w \|_2\|w \|_2)( \|\phi_i\|_{\dot H^\gamma} +1) \bigg)\\
=:&T_{1}+T_2+T_3 .\label{T_i}&
\end{align}
Now we estimate each term in \eqref{T_i}. \ifx We estimate $T_1$ in different scenarios.  When the vanishing order $j$ of the critical point $y_i^\star$ is even, the shear flow profile $u(y)$ is strictly monotone near the critical point. Hence we observe that if $\delta$ is small enough, $|u(y_i^\star)-\lambda|=|u(y_i^\star)-u(y_{i;\lambda})|\geq \delta \nu^{\frac{j+1}{j+1+2\gamma}} $ and $|u(y)-u(y_{i;\lambda})|\leq \frac{1}{100}\delta\nu^{\frac{j+1}{j+1+2\gamma}} $, then $y_{i;\lambda}-y_i^\star$ and $y-y_i^\star$ have the same sign. When treating the $T_1$-term , we apply the properties b), c), d) of the function $\phi_i$, the Taylor expansion of $u$ near the critical point $y_\star$ and the smallness of $\delta$ to obtain 
\begin{align}
\xi_i^2&\bigg|\frac{\sqrt{\phi_i}}{\sqrt{|u-\lambda|}}\bigg|^2=\xi_i^2 \frac{|\phi_i|}{{|u-u(y_{i;\lambda})|}} \\
\leq&C\frac{\|\pa_y\phi_i\|_\infty|y-y_{i;\lambda}|}{\textcolor{red}{|u^{(j+1)}(y_i^\star)\| (y-y_i^\star)^{j+1}-(y_{i;\lambda}-y_i^\star)^{j+1}|}}\mathbf{1}_{|u(y)-u(y_{i;\lambda})|\leq\frac{1}{100} \delta  \nu^{\frac{j+1}{j+1+2\gamma}}}\\
&+C\xi_i^2\delta^{-{1}}\nu^{-\frac{j+1}{j+1+2\gamma}} \mathbf{1}_{|u(y)-u(y_{i;\lambda})|\geq \frac{1}{100}\delta  \nu^{\frac{j+1}{j+1+2\gamma}}}\\
\leq&C\frac{\|\pa_y\phi_i\|_\infty}{|\sum_{\ell=0}^j (y-y_i^\star)^j(y_{i;\lambda}-y_i^\star)^{j-\ell}|}\mathbf{1}_{|u(y)-u(y_{i;\lambda})|\leq\frac{1}{100} \delta  \nu^{\frac{j+1}{j+1+2\gamma}}}\\
&+C\xi_i^2\delta^{-{1}}\nu^{-\frac{j+1}{j+1+2\gamma}} \mathbf{1}_{|u(y)-u(y_{i;\lambda})|\geq \frac{1}{100}\delta  \nu^{\frac{j+1}{j+1+2\gamma}}}\\
\leq& C\delta^{-1}\nu^{-\frac{j+1}{j+1+2\gamma}} .
\end{align}
\myb{In the case when the vanishing order $j$ is odd, we first note that $u(y)\approx u(y_i^\star)+(y-y_i^\star)^{j+1}$ near the critical point by Taylor's theorem. As a result, the equation $\lambda=u(y_{i;\lambda})$ now has two zeros, i.e., $y_{i;\lambda}^+$ and $y_{i;\lambda}^-$. Here  to estimate the expression in \eqref{factor_T1}, we  first fix the sign of $y-y_i^\star$, then choose the $y_{i;\lambda}^\pm$ accordingly, such that $y-y_i^\star$ and $y_{i;\lambda}^\pm-y_i^\star$  have the same sign. Then we use the property of $\phi_i$ at the corresponding $y_{i;\lambda}^\pm$ and similar argument to derive \eqref{factor_T1} to get the estimate.} 
\fi
Combining the estimate \eqref{factor_T1} and the relation \eqref{Real}, we obtain
\begin{align}
T_1\leq & C(u, \delta^{-1} )\nu^{-2\frac{j+1}{j+1+2\gamma}} \|F\|_2^2.\label{T_1}
\end{align} 
Now we estimate $T_2$ term with the relation \eqref{Real}  as follows
\begin{align}
T_2\leq&C(B,\delta^{-1})\nu^{-2\frac{j+1}{j+1+2\gamma}} \|F\|_2^2+\frac{1}{B}\|w \|_2^2.\label{T_2}
\end{align} 
Next we estimate $T_3$. The estimate of $T_3$ is similar to the estimate of $T_6$ \eqref{Case_a_T_6} in the proof of Lemma \ref{lem:case a}.    Combining the property b) of the function $\phi_i$,  the relation \eqref{Real}, Gagliardo-Nirenberg inequality,   and  Young  inequality, we obtain
\begin{align}
T_3
\leq& C(B,u,\delta^{-1},j)\nu^{-2\frac{j+1}{j+1+2\gamma}} |\log(\nu )|^{\frac{16\gamma\al(\gamma)}{2\gamma+1}} \|F\|_2^2+\frac{1}{B}\|w \|_2^2.\label{T_3}
\end{align} 
Combining the estimates \eqref{T_1}, \eqref{T_2}, \eqref{T_3} and the relation \eqref{T_i}, we obtain the estimate 
\begin{align}\label{Case_b_Step_2}
\|w \xi_i\|_{L^2(E_i^c)}^2\leq& C(B,\delta^{-1},u,j)\nu^{-2\frac{j+1}{j+1+2\gamma}} |\log(\nu )|^{\frac{16\gamma\al(\gamma)}{2\gamma+1}}\|F\|_2^2+\frac{1}{B}\|w \|_2^2,\quad \forall i\in \mathcal{I}_{\mathrm{interm}}.
\end{align}This concludes step \#2.

\noindent
\textbf{Step \#3: }Estimation of $\|\xi_i w\|_{L^2(E_i)}^2$ \eqref{defn:E}. 
Applying the area estimate \eqref{Area_E_i_cap_xi_i}, the relation \eqref{Real}, Gagliardo-Nirenberg interpolation inequality and Young inequality, we estimate the $L^2$ contribution  as follows: 
\begin{align}
\|w\xi_i \|_{L^2(E_i)}^2\leq &{|E_i|}\|w \|_{L^\infty(\Torus)}^2\\
\leq&C \delta^{\frac{1}{j+1}} \nu^{\frac{1}{j+1+2\gamma}}\||\pa_y |^\gamma w \|_{L^2(\Torus)}^{\frac{1}{\gamma}}\|w \|_{L^2(\Torus)}^{2 -\frac{1}{\gamma}}+C\delta^{\frac{1}{j+1}} \nu^{\frac{1}{j+1+2\gamma}}\|w \|_{L^2(\Torus)}^2\\
\leq &C(B) \nu^{-\frac{j+1}{j+1+2\gamma}}\nu\||\pa_y|^\gamma w \|_{L^2(\Torus)}^2+\left(\frac{1}{B}+C\delta^{\frac{1}{j+1}} \nu^{\frac{1}{j+1+2\gamma}}\right)\|w \|_{L^2(\Torus)}^2\\ 
\leq&C(B)\nu^{-\frac{j+1}{j+1+2\gamma}}  \|F\|_{L^2(\Torus)}\|w \|_{L^2(\Torus)}+\left(\frac{1}{B}+C\nu^{\frac{1}{j+1+2\gamma}}\right)\|w \|_{L^2(\Torus)}^2\\
\leq&C(B)\nu^{-2\frac{j+1}{j+1+2\gamma}} \|F\|_{L^2(\Torus)}^2+\left(\frac{1}{B}+C\nu^{\frac{1}{j+1+2\gamma}}\right)\|w \|_{L^2(\Torus)}^2.\label{wk_L2Ei}
\end{align}
 Combining \eqref{wk_L2Ei} and estimate \eqref{Case_b_Step_2} implies \eqref{Goal_i_by_i} for $i\in \mathcal{I}_{\mathrm{interm}}$. 
\end{proof}
Now we consider $i\in \mathcal{I}_{\mathrm{far}}$ \eqref{Case c}. 
\begin{lem} \label{lem:case c}
Assume the conditions in Theorem \ref{thm:semigroup_est_PS}. For $i\in \mathcal{I}_{\mathrm{far}}$ \eqref{Case c}, the estimate  \eqref{Goal_i_by_i}   holds for $\nu_0(u)>0$ small enough. 
\end{lem}
\begin{proof} Combining the assumption $\mathrm{dist}(y_i^\star,\{z|u(z) =\lambda\})\geq  3r_i$ and the fact that the derivative $u'$ is away from zero in the region $B(y_i^\star, 3r_i)\backslash  \mathrm{support}(\xi_i^2)$, we have that $|u(y)-\lambda|\geq \frac{1}{C(u)}$ for $y\in  \mathrm{support}(\xi_i^2)=B(y_i^\star,2r_i) $. Now we apply the relation \eqref{Imaginary_with_f} with $g=\xi_i^2$, the product estimate \eqref{product_rule}, the Gagliardo-Nirenberg interpolation inequality and Young inequality to obtain that
\begin{align}
\|w  \xi_i\|_2^2&\leq C\bigg(\|F\|_2\|w \xi_i\|_2+ {\nu}\||\pa_y|^\gamma w \|_2(\||\pa_y|^\gamma w\|_2+\|w\|_2)\|\xi_i\|_\infty\\
&\quad\quad\quad\quad\quad\quad\quad+ {\nu}\||\pa_y|^\gamma w \|_2\|w \|_\infty(\||\pa_y|^\gamma(\xi_i^2)\|_2+\|\xi_i^2\|_2)\bigg)\frac{1}{ \min_{y\in  \mathrm{support}(\xi_i^2)}|u(y)-\lambda|}\\
&\leq {C(u,\delta^{-1},B)} \|F\|_2^2+\frac{1}{B}\|w \|_2^2+ {C(u, \|\xi_i\|_{H^2})} \nu(\||\pa_y|^\gamma w \|_2^2+\|w \|_2^2)\\
&\leq C(u,\delta^{-1},B) \|F\|_2^2+\left(\frac{1}{B}+C(u)\nu\right)\|w \|_2^2
\end{align} for any positive constant $B$. This implies  the estimate \eqref{Goal_i_by_i}.
\end{proof}

Finally, we estimate $\|\xi_0 w \|_2^2. $ \begin{lem} \label{lem: 0}
Assume the condition in Theorem \ref{thm:semigroup_est_PS}. Then the  following estimate holds for any positive constant $B>1$ and $\nu_0(u)>0$ small enough,
\begin{align}
\|\xi_0 w\|_2^2\leq C(B,\delta^{-1}, u)\nu^{-\frac{2}{2\gamma+1}}|\log\nu|^{ 16\gamma \al(\gamma)}\|F\|_2^2+\left(\frac{1}{B}+C(u,\delta^{-1})\nu^{\frac{1}{2\gamma+1}}\right)\|w\|_{L^2(\Torus)}^2.\label{T_0_est}
\end{align}
\end{lem}
\begin{proof}

Note that on the support of $\xi_0$, the function $u(y)$ intersects the value $\lambda$ at different points $\{y_h^\dagger\}_{h=1}^M, \quad M\leq N+1$, and near each intersection, the derivative of $u$ is away from zero, i.e., $\|u'(y_h^\dagger)\|_\infty\geq \frac{1}{C(u){}}>0$. We use the transition function trick again. To this end, we consider the following  transition function $\psi_0$:
\begin{align}
a)\quad &\psi_0 (y)=\mathrm{sign}(u(y)-\lambda),\quad |y-y_h^\dagger|\geq  \delta  \nu^{\frac{1}{2\gamma+1}};\\
\label{Regularity_of_psi_0}b)
 \quad &\|\psi_0\|_\infty\leq 1 , \quad\|\pa_y\psi_0\|_\infty\leq C(u){ \delta^{-{1}} \nu^{-\frac{1}{2\gamma+1}}},\quad\|\psi_0\|_{\dot H^1}\leq C(u){ \delta^{-1/2} \nu^{-\frac{1}{2(2\gamma+1)}}},\\
\quad &\||\pa_y|^{1/2}\psi_0\|_2\leq C(u,\delta^{-1})|\log \nu |,\quad \|\psi_0\|_{\dot H^2}\leq C(u){\delta^{-3/2} \nu^{-\frac{3}{2(2\gamma+1)}}};\\
c) \quad& \text{Near the zero points } y_h^\dagger, \quad C^{-1}|\pa_y\psi_0(y_h^\dagger)(y-y_h^\dagger)|\leq|\psi_0(y)|\leq C|\pa_y\psi_0(y_h^\dagger)(y-y_h^\dagger)|.
\end{align}
Now we focus on one intersection and comment that other intersections are similar. We test the equation \eqref{Resolvent} with $\psi_0\xi_0^2 \overline{w }$ and obtain the following relation 
\begin{align}
 \int_{\Torus}|(u-\lambda)\psi_0|\xi_0^2|w |^2dy=\int_{\Torus}(u-\lambda)\psi_0\xi_0^2| w|^2dy =\mathrm{Im}\int_\Torus F\psi_0\xi_0^2 \overline{w }dy-\nu\mathrm{Im}\int_\Torus |\pa_y|^\gamma w |\pa_y|^\gamma(\psi_0 \xi_0^2\overline{w })dy.\label{Real_part_T_0}
\end{align} 
We further decompose the domain into the following two component: \\
a) $E_0:=\{y|\,\exists h\in\{1,..., M\} \text{ such that } |y-y_h^\dagger|\leq \delta \nu^{\frac{1}{2\gamma+1}}  \}$; b) $E_0^c$. 
To estimate the solution on $E_0$, we apply the  Gagliardo-Nirenberg interpolation inequality inequality and $\gamma>1/2$ to obtain:
\begin{align}
\|\xi_0 w \|_{L^2(E_0)}^2\leq &{|E_0|} \|w \|_{L^\infty(\Torus)}^2\\
\leq &C  \nu^{\frac{1}{2\gamma+1}}\||\pa_y|^\gamma w \|_{L^2(\Torus)}^{\frac{2}{2\gamma}}\|w \|_{L^2(\Torus)}^{2\frac{2\gamma-1}{2\gamma}}+C \nu^{\frac{1}{2\gamma+1}}\|w \|_{L^2(\Torus)}^2\\
\leq &C(B) \nu^{\frac{2\gamma}{2\gamma+1}}\||\pa_y|^\gamma w \|_{L^2(\Torus)}^2+\left(\frac{1}{B}+C \nu^{\frac{1}{2\gamma+1}}\right)\|w \|_{L^2(\Torus)}^2\\
\leq&C(B)\nu^{-\frac{2}{2\gamma+1}} \|F\|_{L^2(\Torus)}^2+ \left(\frac{1}{B}+C \nu^{\frac{1}{2\gamma+1}}\right)\|w \|_{L^2(\Torus){}}^2.
\label{L2_w_F}
\end{align}
Now we estimate the $E_0^c$ component of the solution. To this end, we collect the relation \eqref{Real_part_T_0}, the fact that $|u'|(y)$ is bounded below on the set $E_0$, and the fact that $|y-y_h^\dagger|\geq \delta \nu^{\frac{1}{2\gamma+1}}$ on $E_0^c$. Then we apply the product estimate \eqref{product_rule}, the Gagliardo-Nirenberg interpolation inequality to estimate the $E_0^c$ part as follows,
\begin{align}
\|\xi_0 w \|
&_{L^2(E_0^c)}^2\leq C(u)\delta^{-1}\left(\min_{y\in E_0}|u'(y)|\right)^{-1}\nu^{-\frac{1}{2\gamma+1}} \int_{E_0^c}|u-\lambda \|\psi_0\|\xi_0|^2|w |^2dy\\
\leq &C(u)\delta^{-1}\left(\min_{y\in E_0}|u'(y)|\right)^{-1}\nu^{-\frac{1}{2\gamma+1}} \int_{\Torus }(u-\lambda )\psi_0\xi_0^2|w |^2dy\\
\leq&C(\delta^{-1},u,\|\xi_0\|_{H^2 })\nu^{-\frac{1}{2\gamma+1}}\bigg( \norm{\frac{F\sqrt{|\psi_0|}\xi_0}{\sqrt{|u-\lambda|}}}_{L^2}^2+ {\nu} \||\pa_y|^\gamma w \|_{L^2 }(\||\pa_y|^\gamma w\|_{L^2 }+\|w\|_{L^2})\|\psi_0\|_{L^\infty}\\
&\quad\quad \quad+ {\nu} \left(\||\pa_y|^\gamma w \|_2^{\frac{2\gamma+1}{2\gamma}}\|w \|_2^{\frac{2\gamma-1}{2\gamma}}+\||\pa_y|^\gamma w \|_2\|w \|_2\right)(\|\psi_0\|
_{\dot H^\gamma}+1)\bigg)\\
=:&\sum_{i=1}^3 T_{ i}.&\label{T_0i}
\end{align}
We estimate each term in the expression \eqref{T_0i}. Before estimating the $T_{1}$ term, we apply property c) of the function $\psi_0$, the fact that the derivative of $u$ does not vanish in the set $E_0\cap \mathrm{support}(\xi_0)$ to derive the following relation
\begin{align}
\xi_0^2(y) \frac{|\psi_0(y)|}{{|u(y)-\lambda|}} \leq & \xi_0^2(y)\mathbf{1}_{E_0}(y) \frac{\|\pa_y\psi_0\|_\infty |y-y_h^\dagger|}{{|y-y_h^\dagger|}\min_{z\in E_0}|u'(z)|} +\xi_0^2(y)\mathbf{1}_{ E_0^c}(y)\frac{1}{\min_{z\in E_0}|u'(z)|}\delta^{-1}\nu^{-\frac{1}{2\gamma+1}}\\
\leq& C(u)\delta^{-1}\nu^{-\frac{1}{2\gamma+1}},\quad \forall y\in \Torus.
\end{align}
The estimate yields  that 
\begin{align}\label{Case_0_T_1}
T_{ 1}\leq C(\delta^{-1},u)\nu^{-\frac{2}{2\gamma+1}} \|F\|_2^2.
\end{align}
Now we estimate $T_2$ in the decomposition using H\"older inequality, Young's inequality and the estimate \eqref{Real}, as follows
\begin{align}
T_{ 2}
\leq &C(B,\delta^{-1},u)\nu^{-\frac{2}{2\gamma+1}} \|F\|_2^2+\frac{1}{B}\|w \|_2^2.\label{Case_0_T_2}
\end{align} 
\ifx Now we estimate $T_3$ in \eqref{T_0i} using H\"older inequality, Young's inequality, the fact that $\gamma\geq0$, and the estimate \eqref{Real} as follows
\begin{align}
T_{3}\leq &C(\delta^{-1},u)\nu^{-\frac{1}{2\gamma+1}+1} \||\pa_y|^\gamma w \|_2\|w \|_2\|\psi_0\|_\infty\\
\leq&C(B,\delta^{-1},u)\nu^{-\frac{2}{2\gamma+1}+1} \|F\|_2\|w \|_2+\frac{1}{B} \|w \|_2^2\\
\leq & C(B,\delta^{-1},u)\nu^{-\frac{4}{2\gamma+1}+2} \|F\|_2^2+\frac{1}{B}\|w \|_2^2\\
\leq& C(B,\delta^{-1},u)\nu^{-\frac{2}{2\gamma+1}} \|F\|_2^2+\frac{1}{B}\|w \|_2^2,
\end{align}
which is consistent with \eqref{T_0_est}. 
\fi 
Finally, we estimate $T_{3}$ in \eqref{T_0i}. 
We estimate the $H^\gamma$-norm of $\psi_0$. By the Gagliardo-Nirenberg interpolation inequality and the regularity of $\psi_0$ \eqref{Regularity_of_psi_0}, we have the following estimate 
\begin{align}
\|\psi_0\|_{\dot H^{\gamma}}\leq&C(\delta^{-1})|\log \nu|^{2\al(\gamma)}\nu^{-\frac{2\gamma-1}{2(2\gamma+1)}},\quad\gamma\in(1/2,2].
\end{align} 
Combining these two estimate together with H\"older inequality, Young's inequality and the estimate \eqref{Real}, we obtain that for $\gamma\in(1/2,2]$, 
\begin{align}
T_{ 3}\leq &C(B,\delta^{-1},u)|\log \nu|^{\frac{8\gamma\al(\gamma)}{2\gamma+1}}\nu^{\frac{2\gamma}{2\gamma+1}}\||\pa_y|^\gamma w\|_2^2+C(B,\delta^{-1},u)|\log \nu|^{4\al(\gamma)}\nu\||\pa_y|^\gamma w\|_2^2+\frac{1}{B}\|w\|_2^2\\
\leq&C(B,\delta^{-1},u)\nu^{-\frac{2}{2\gamma+1}}|\log \nu|^{{16\gamma\al(\gamma)}} \|F\|_2^2+\frac{1}{B}\|w \|^2_2.\myb{(Check?)}\label{Case_0_T_3}
\end{align}
Combining the estimates \eqref{T_0i}, \eqref{Case_0_T_1}, \eqref{Case_0_T_2}, \eqref{Case_0_T_3} above and \eqref{L2_w_F} yields the estimate \eqref{T_0_est}.  
\ifx\begin{align} 
\|\xi_0w\|_2^2\leq C(B,\delta^{-1},u)\nu^{-\frac{2}{2\gamma+1}} |\log \nu|^{16\gamma\al(\gamma)}\|F\|_2^2+\left(\frac{1}{B}+C\nu^{\frac{1}{2\gamma+1}}\right) \|w \|_2^2.\label{T_0}
\end{align}\fi 
\end{proof}
\begin{proof}[Proof of Proposition \ref{pro:resolvent_estimate}]  
We combine Lemma \ref{lem:case a}, Lemma \ref{lem:case b}, Lemma \ref{lem:case c}, and Lemma \ref{lem: 0}, to obtain that if $\min_{y\in\Torus}|\lambda- u(y)|\leq \delta(u)\nu^{\frac{j_m +1}{j_m +1+2\gamma}}$ for $\delta(u)$ chosen in \eqref{Choice_of_delta},  the following estimate holds
\begin{align}
\|w \|_2^2=\sum_{i=0}^N\|w \xi_i\|_2^2\leq &
C(B,\delta^{-1},u,N,\{j_i\}_{i=1}^N)\nu^{-2\frac{j_m +1}{j_m +1+2\gamma}} |\log(\nu )|^{16\gamma\al(\gamma)}\|F\|_2^2\\
&+CN\left(\frac{1}{B}+C(u,\delta^{-1},\{j_i\}_{i=1}^N)\nu^{\frac{1}{j_m +1+2\gamma}}|\log\nu|^{4\al(\gamma)}\right)\|w  \|_2^2.
\end{align} 
If we choose $B^{-1}$, and then $\nu_0 $ small enough, then for $0<\nu\leq\nu_0$, the following estimate holds
\begin{align}\label{Choice_of_B_delta_nu_0}
\|w \|_2^2\leq & C(u)\nu^{-2\frac{j_m +1}{j_m +1+2\gamma}} |\log(\nu )|^{16\gamma \al(\gamma)}\|F\|_2^2.
\end{align}
Combining the $L^2$-estimate with the relation \eqref{Real}
 yields the conclusion \eqref{Goal}.  Recalling Lemma \ref{lem:away_from_range}, this completes the proof of Proposition \ref{pro:resolvent_estimate}.  
\end{proof}

\ifx 
\appendix \myb{I think it is possible to use the profile to prove $L^\infty$-ED without loss in $log\nu$? But we might have the following $\|\wh\eta_k(t)\|_{L^\infty}\leq C(\nu^{-1}) e^{-\delta_{ED}d(\nu)t}$.} For $0<\gamma\leq 1/2$, we might need the $|\pa_x|^{2\gamma$ and we might need the comparison principle to guarantee that the $L^\infty$-norms are bounded. 
\section{Enhanced dissipation of the $L^\infty$-norm}
Now, we prove the following lemma which provides $L^\infty$-estimates for the passive scalar solutions. \textcolor{red}{The proof requires the maximum principle.}
\begin{lem}[$L^\infty$-decay of the passive scalar solution]\label{Lem:Linfty}
Consider solutions $n$ to the passive scalar equation \eqref{EQ:Fractional_PS}. Assume conditions in Theorem \ref{thm:semigroup_est_PS_full} and assume $\gamma\in(1/2,1]$. Then if $\nu$ is small enough, there exist constants $c\in(0,1),\,C$ such that the following estimate holds
 \begin{align}
\| n(t)\|_\infty\leq C \|n_{0} \|_\infty e^{-c\delta d(\nu)|\log\nu|^{-\al-1}t},\quad \forall t\in[0,\infty).
\end{align}
\end{lem}
\begin{proof}
We derive an  $L^2$-$L^\infty$-estimate of the passive scalar semigroup $S_{s;s+t}$. Consider the time interval $[s, s+c^{-1}\delta^{-1}d(\nu)^{-1}|\log \nu|^{\al+1}]$, where $c\in(0,1)$ denotes a universal constant. First we  prove the following estimate for passive scalar equation 
\begin{align}\label{Nash}
\| S_{s;s+t}n(s)\|_\infty\leq \frac{C }{(\nu t)^{\frac{1}{2\gamma}}}\|n(s)\|_2.
\end{align} 
The proof of this estimate \eqref{Nash} is a combination of Nash inequality and a duality argument. We follow the  the proof of equation (4.4) in \cite{IyerXuZlatos}, \cite{CKRZ08}, Lemma 5.4 in \cite{Zlatos2010}.   
Without loss of generality, we assume that $s=0$. 
To prove \eqref{Nash}, we estimate the time evolution of the $L^2$-norm with Nash inequality as follows:
\begin{align}
\frac{d}{dt}\frac{1}{2}\|n\|_2^2\leq& -\nu \||\na|^\gamma n\|_2^2
\leq-\frac{\nu\|n\|_2^{2\gamma+2}}{C_N\|n\|_1^{2\gamma}}\leq-\frac{\nu\|n\|_2^{2\gamma+2}}{C_N\|n_{0}\|_1^{2\gamma}}.
\end{align}
Here the $L^1$-norm of $n$ is non-expansive because we can consider the solutions to \eqref{EQ:Fractional_PS} subject to the positive and negative part of the initial data, i.e., $\max\{\pm n_0,0\}.$ \textcolor{red}{Since both of them are positive and have decaying $L^1$-norms and $n$ is the sum of these two solutions, we have that the $L^1$-norm of $n(t)$ is less than the $L^1$-norm $\| n_0\|_1$.} Next we directly solve the ordinary differential inequality subject to infinity initial data and obtain that there exists a universal constant $C$ such that the following estimate holds
\begin{align}\|n(t)\|_2\leq\frac{C}{\gamma^{\frac{1}{2\gamma}}(\nu t)^{\frac{1}{2\gamma}}}\|n_0\|_1.\label{L1toL2est}
\end{align}
Next we consider the dual solution of the equation \eqref{EQ:Fractional_PS}
\begin{align}\label{Dual_equation}
\pa_s \varphi^t=-\nu(-\de_x)^{\gamma} \varphi^t-\nu(-\de_y)^{\gamma}\varphi^t + u( y ) \pa_x \varphi^t,\quad \varphi^t(0,x,y)=\varphi_0(x,y),\quad \int_{\Torus}\varphi_0(x,y)dV=0.
\end{align}
Next we prove the following  dual  relation
\begin{align}
\int_{\Torus^2}n(t,x,y)\varphi_0(x,y)dV=\int_{\Torus^2}n_0(x,y)\varphi^t(t,x,y)dV.\label{dual}
\end{align}To this end , we consider the pairing
\begin{align}
\int_{\Torus^2}n(t-s,x,y)\varphi^t(s,x,y)dV,\quad \forall s\in[0,t].
\end{align}  
Direct $s$-derivative yields that
\begin{align}
&\frac{d}{ds}\int_{\Torus^2}n(t-s,x,y)\varphi^t(s,x,y)dV\\
&=\int_{\Torus^2} (\nu(-\de_x)^\gamma +\nu(-\de_y)^{\gamma}+u(y ) \pa_x )n(t-s,x,y)\varphi^t(s,x,y)dV\\
&\quad+\int_{\Torus^2}n(t-s,x,y)(-\nu(-\de_x)^{\gamma}-\nu(-\de_y)^{\gamma}+u( y )  \pa_x )\varphi^t(s,x,y)dV=0.& 
\end{align}
Hence the pairing is a constant. Now by taking $s=0$ and $s=t$, we obtain the relation \eqref{dual}. By the same argument as the one yielding \eqref{L1toL2est}, we have the following estimate for the $\varphi^t(s)$
\begin{align}
\|\varphi^t(t)\|_2\leq  \frac{C}{(\nu t)^{\frac{1}{2\gamma}}}\|\varphi_0\|_1.
\end{align}
Now we can use the  relation \eqref{dual} to estimate the $L^\infty$-norm of $n(t)$
\begin{align}
\bigg|\int n(t,x,y)\varphi_0(x,y) dxdy \bigg|=&\bigg|\int n_{0}(x,y)\varphi^t(t,x,y)dxdy\bigg|\leq \|n_{0}\|_{L^2(\Torus^2)}\|\varphi^t( t,\cdot)\|_{L^2(\Torus^2)}\\
\leq&\frac{C}{(\nu t)^{\frac{1}{2\gamma}}}\|\varphi_0\|_{L^1(\Torus^2)}\|n_{0}\|_{L^2(\Torus^2)}.
\end{align}
By the dual characterization of $L^\infty$-norm, we have that
\begin{align}
\|n(t)\|_{L^\infty(\Torus^2)}\leq\frac{C}{(\nu t)^{\frac{1}{2\gamma}}}\|n_{0}\|_{L^2(\Torus^2)}.
\end{align} 

Now we decompose the interval $[s, s+c^{-1}\delta^{-1}d(\nu)^{-1}|\log \nu|^{\al+1}]$ into two equal-length sub-intervals and  apply the following estimate:
\begin{align}
\|S&_{s,s+c^{-1}\delta^{-1}d(\nu)^{-1}|\log \nu|^{\al+1}}n(s)\|_\infty\\
=&\|S_{s+\frac{1}{2}c^{-1}\delta ^{-1}d(\nu)^{-1}|\log \nu|^{\al+1},s+c^{-1}\delta ^{-1}d(\nu)^{-1}|\log \nu|^{\al+1}}S_{s,s+\frac{1}{2}c^{-1}\delta ^{-1}d(\nu)^{-1}|\log \nu|^{\al+1}}n(s)\|_\infty\\
\leq &\frac{C }{(\nu \frac{1}{2}c^{-1}\delta ^{-1}d(\nu)^{-1}|\log \nu|^{\al+1})^{\frac{1}{2\gamma}}}\|S_{s,s+\frac{1}{2}c^{-1}\delta ^{-1}d(\nu)^{-1}|\log \nu|^{\al+1}}n(s)\|_2\\
\leq&\frac{C}{(c^{-1}\delta ^{-1}\nu d(\nu)^{-1}|\log\nu|^{\al+1})^{\frac{1}{2\gamma}}}\|n(s)\|_2 {e^{-\frac{1}{2}\delta  d(\nu)|\log \nu|^{-\al}c^{-1}\delta ^{-1}d(\nu)^{-1}|\log \nu|^{\al+1}}}\\
\leq &\frac{1}{64}\|n(s)\|_2{\leq \frac{1}{32}\pi \|n(s)\|_\infty}.\label{L2_Linfty_semigroup}
\end{align}
In the last line, we choose $c$ and then $\nu$ small enough compared to universal constants so that the coefficient is small. We further note that the $L^\infty$-norm of $n$ is dissipative along the dynamics. To conclude, we iterate the argument on consecutive intervals to derive the estimate. 
\end{proof}
\fi
\bibliographystyle{abbrv}
\bibliography{nonlocal_eqns,JacobBib,SimingBib}
\end{document}